\documentclass[11pt]{amsart}

\usepackage[a4paper,margin=30mm]{geometry}
\usepackage{amsmath,amssymb,amsthm,mathtools}
\usepackage{graphicx}
\usepackage{enumitem}
\usepackage{hyperref}
\usepackage[nameinlink,capitalise]{cleveref}
\hypersetup{hidelinks}
\crefname{assumption}{Assumption}{Assumptions}
\Crefname{assumption}{Assumption}{Assumptions}
\usepackage{mathrsfs}

\newtheorem{theorem}{Theorem}[section]

\newtheorem{corollary}[theorem]{Corollary}

\theoremstyle{definition}

\newtheorem{remark}[theorem]{Remark}

\title[Exact Rayleigh Reduction]
{Exact Rayleigh Reduction for Direct Detection of Turning Bifurcations}

\author{Yavdat Il'yasov}
\address{Institute of Mathematics with Computing Centre, Ufa Federal Research Centre of the Russian Academy of Sciences, Ufa, Russia}
 \email{ilyasov02@gmail.com}


\begin{document}

\begin{abstract}
	We introduce an exact Rayleigh reduction for the direct construction and
	detection of turning bifurcations in nonlinear parameter-dependent
	equations. A generalized Rayleigh functional ordinarily provides only a
	scalar necessary constraint on solutions. We show that, along a suitable
	one-parameter transformation orbit $u=\mathcal G_s\phi$, it can instead
	become the exact parameter law
	$\lambda=\mathcal R(\mathcal G_s\phi)$ of a genuine solution branch.
	Turning points are then found from nondegenerate critical points of this
	one-dimensional Rayleigh profile. The branch tangent automatically yields
	a kernel direction of the state linearization, and under the usual
	Fredholm, kernel-simplicity, and parameter-transversality assumptions the
	detected point is a simple fold.
	
	For Kirchhoff equations, amplitude scaling on bounded domains and spatial
	dilation on $\mathbb R^N$ yield explicit exact branches, global parameter
	thresholds, and countable hierarchies of turning values. For the generalized
	Kirchhoff law $M_b(A)=a+bA^\theta$, the branch geometry is governed by the
	dimension--homogeneity index $\theta(N-2)-2$: an interior turning
	bifurcation occurs exactly when $\theta(N-2)>2$. This gives a universal
	bifurcation trichotomy that is independent of the particular
	Berestycki--Lions nonlinearity.
\end{abstract}
\maketitle

\section{Introduction}
\label{sec:introduction}

A central problem in nonlinear analysis is to determine the parameter values
at which the structure of solutions changes. Such values mark the appearance,
merging, or disappearance of solution branches and arise naturally in
nonlinear elliptic equations, variational problems, and models with nonlocal
interactions. Classical bifurcation and singularity theories provide powerful
tools for describing the local structure near such points; see, for example,
\cite{arnold,GolubitskySchaeffer1985,seydel}.

Throughout the paper, by a \emph{turning bifurcation} we mean a point at
which two nearby solution states coalesce as the parameter reaches a
critical value and the corresponding solution branch turns with respect to
the parameter. A classical simple fold is obtained when the usual additional
Fredholm and transversality conditions hold.

There is, however, a complementary constructive question: can such a
critical parameter be detected directly from the nonlinear equation before
the singular point itself is known? This is particularly relevant when the
solution set is not explicitly available and one seeks not only the critical
value, but also the mechanism responsible for the change in solution
geometry.

The purpose of this paper is to develop a generalized Rayleigh approach to
this problem. Consider
\[
F(\lambda,u)=0,
\qquad
F:\mathbb R\times X\to Y,
\]
and suppose that every nontrivial solution satisfies a parameter relation
\begin{equation}\label{eq:intro-parameter-Rayleigh}
	\lambda=\mathcal R(u),
\end{equation}
where $\mathcal R$ is a generalized Rayleigh functional. In general,
\eqref{eq:intro-parameter-Rayleigh} is only a necessary scalar constraint:
\[
\lambda=\mathcal R(u)
\qquad\not\Longrightarrow\qquad
F(\lambda,u)=0.
\]

The main concept introduced here is an \emph{exact Rayleigh reduction}.
We seek a one-parameter transformation orbit
\[
u(s)=\mathcal G_s\phi
\]
and distinguished profiles $\phi$ for which the Rayleigh relation becomes
the actual parameter law of an exact solution branch:
\begin{equation}\label{eq:intro-exact-orbit}
	F\bigl(
	\mathcal R(\mathcal G_s\phi),
	\mathcal G_s\phi
	\bigr)=0.
\end{equation}
For every such profile,
\begin{equation}\label{eq:intro-Rayleigh-profile}
	\lambda_\phi(s)
	=
	\mathcal R(\mathcal G_s\phi)
\end{equation}
is the parameter along the exact branch
\[
s\longmapsto
\bigl(
\lambda_\phi(s),\mathcal G_s\phi
\bigr).
\]
Thus an infinite-dimensional bifurcation problem is reduced, along an exact
orbit, to the geometry of the one-dimensional Rayleigh profile
$\lambda_\phi$. Its nondegenerate critical points detect turning
bifurcations of the exact branch.

In applications, the exactness condition
\eqref{eq:intro-exact-orbit} typically reduces to a normalized profile
equation whose solutions select the admissible exact branches. A turning
bifurcation candidate is then located directly from
\begin{equation}\label{eq:intro-critical-Rayleigh}
	\lambda_\phi'(s)=0.
\end{equation}
At a nondegenerate critical point $s_*$, differentiation of the exact branch
identity gives
\[
F_u(\lambda_*,u_*)u'(s_*)=0,
\qquad
u_*=\mathcal G_{s_*}\phi,
\qquad
\lambda_*=\lambda_\phi(s_*).
\]
Hence the tangent to the explicitly constructed branch provides a kernel
direction of the state linearization. Under additional kernel-simplicity,
Fredholm, and parameter-transversality assumptions, the detected turning
bifurcation is a simple fold. Whenever the operator realization is of class
$C^2$, this agrees with the classical Lyapunov--Schmidt simple-fold
criterion.

This gives generalized Rayleigh quotients a different role. Rather than
merely identifying distinguished parameter levels, they may reconstruct
exact portions of the nonlinear solution set itself. Multiplicity, turning
thresholds, and branch hierarchies can then be read directly from the
geometry of explicit Rayleigh profiles. Exact Rayleigh reduction and local
bifurcation theory therefore play complementary roles: the former constructs
the relevant branch and locates its critical parameter, while the latter
provides a local classification of the singularity.

The approach originates from nonlinear Rayleigh quotient methods developed
for Nehari and Pohozaev manifolds
\cite{Ilyasov2002Bif,IlyasovEgorov2010, Ilyasov2017, Ilyasov2022LevelSets}.
There, critical levels of generalized Rayleigh functionals were used to
characterize distinguished parameter values and extremal thresholds.
Related minimax constructions for singular parameter values were developed
through extended Rayleigh quotient methods
\cite{Ilyasov2007,Ilyasov2021,Ilyasov2026Minimax}.
The present paper takes a further step: under an appropriate transformation,
the Rayleigh functional becomes the parameter law of an exact solution
branch whose geometry reveals the bifurcation structure.

We develop this mechanism for Kirchhoff equations. These equations originate
in Kirchhoff's model for the transverse vibrations of a stretched elastic
string, in which the tension depends on the global deformation of the string
and therefore introduces a genuinely nonlocal feedback
\cite{Kirchhoff1883,Ma2005}. Mathematically, the Kirchhoff coefficient couples
a local differential operator to a global energy of the state. This structure
is particularly well suited to exact Rayleigh reduction, since the global
energy interacts naturally with scaling transformations. The method then
isolates the competition between this nonlocal feedback and the homogeneity
of the underlying scaling. In this sense, Kirchhoff equations serve here as
a model class for the reduction mechanism rather than as its endpoint.

Two different reductions are considered: amplitude scaling on bounded
domains and spatial dilation on $\mathbb R^N$.

On bounded domains, the method remains exact in the presence of a positive
spatial weight $k(x)$. For the weighted Kirchhoff problem, the Nehari
Rayleigh functional combined with amplitude scaling generates exact
branches from solutions of a weighted normalized profile equation. The
resulting Rayleigh geometry yields explicit turning bifurcation values, a
global upper existence threshold, and a countable hierarchy of exact
branches. Remarkably, the spatial weight changes the profile-dependent
bifurcation levels but leaves the critical amplitude of every exact branch
unchanged.

On $\mathbb R^N$, spatial dilation leads naturally to a Pohozaev Rayleigh
functional and reduces the Kirchhoff equation to the nonlinear scalar-field
equation
\[
-\Delta\phi=g(\phi).
\]
Every normalized profile generates an exact spatial branch, while the
Nehari representation, when applicable, gives the same parameter law along
that branch. The resulting reduction separates the scalar-field profile
from the geometry induced by the Kirchhoff scaling.

For some particular classes of Kirchhoff equations, parameter thresholds
associated with changes in the number of solutions have previously been
obtained; see
\cite{LiangLiShi2017,LiuLiaoPanTang2022,Silva2019}.
Scalar rescaling relations underlying positive solutions, together with the
corresponding two--one--zero multiplicity pattern, also appear in the
Kirchhoff literature; see
\cite{Azzollini2012,Azzollini2015,HuangLiuWuProc2016}.
These works provide important existence, threshold, and multiplicity
results, and in some cases describe the behavior of solutions near critical
parameter values. The present approach addresses the bifurcation structure
more directly: critical parameter values are located together with the
associated turning bifurcation points, the resulting solution branches are
constructed globally through an exact Rayleigh reduction, and the same
construction yields a countable hierarchy of exact bifurcation branches.
Moreover, the mechanism extends beyond the standard Kirchhoff law.

Indeed, for the generalized Kirchhoff law
\[
M_b(A)=a+bA^\theta,
\]
the exact spatial reduction yields
\[
b_\phi(s)
=
\frac{1-as^2}{A(\phi)^\theta}
s^{\theta(N-2)-2}.
\]
Thus the branch geometry is governed by the dimension--homogeneity index
\[
\theta(N-2)-2.
\]
In particular,
\[
\boxed{
	\theta(N-2)>2
}
\]
is the sharp condition for the occurrence of an interior turning
bifurcation. The supercritical, critical, and subcritical cases
$\theta(N-2)>2$, $\theta(N-2)=2$, and $\theta(N-2)<2$ lead to three
distinct branch geometries. The familiar dimensional distinction for the
standard Kirchhoff equation is therefore a special case of a more general
dimension--homogeneity balance.

The same reduction produces a countable hierarchy of exact spatial
branches. Their turning values are ordered inversely by the energies of the
normalized scalar-field profiles and accumulate at zero, while a
ground-state profile generates the maximal turning value and the global
upper existence threshold. Thus the exact Rayleigh reduction describes
simultaneously the turning mechanism, the global threshold, and the
multiplicity structure.

The paper is organized as follows.
Section~\ref{sec:constraint-to-exact} develops the abstract exact Rayleigh
reduction and its relation to classical local fold theory.
The subsequent sections apply the method to weighted bounded-domain and
whole-space Kirchhoff equations and study the corresponding exact branches,
turning bifurcation thresholds, and branch hierarchies.

\section{Exact Rayleigh reduction and direct bifurcation detection}
\label{sec:constraint-to-exact}

We develop here the abstract mechanism underlying the Rayleigh approach to
bifurcation. The starting point is a generalized Rayleigh representation of
the parameter, but the essential new step is to combine it with a
one-parameter transformation so that the resulting scalar Rayleigh profile
becomes the actual parameter law of an exact solution branch.

Let
\[
F(\lambda,u)=0,
\qquad
F:\mathbb R\times X\to Y,
\]
where $X$ and $Y$ are Banach spaces. Suppose that every nontrivial solution
satisfies a scalar constraint
\[
\mathcal C(\lambda,u)=0,
\]
which can be resolved with respect to the parameter:
\begin{equation}\label{eq:Rayleigh-representation}
	\mathcal C(\lambda,u)=0
	\qquad\Longleftrightarrow\qquad
	\lambda=\mathcal R(u).
\end{equation}
We call $\mathcal R$ the generalized Rayleigh functional associated with
the constraint.

Such representations arise naturally from Nehari, Pohozaev, energy, and
other scalar identities. They form the basis of nonlinear generalized
Rayleigh quotient methods for detecting distinguished parameter levels
\cite{Ilyasov2017,Ilyasov2022LevelSets}. By themselves, however, they
describe only constraint geometry:
\[
\lambda=\mathcal R(u)
\qquad\not\Longrightarrow\qquad
F(\lambda,u)=0.
\]
The purpose of the reduction below is to identify situations in which this
implication becomes exact along a suitable transformation orbit.

\subsection{Orbital Rayleigh profiles and exact reduction}
\label{subsec:orbital-exact-Rayleigh}

Let
\[
\mathcal G_s:X\to X,
\qquad s>0,
\qquad
\mathcal G_1=\mathrm{Id},
\]
be a one-parameter family of transformations, and let
\[
\mathcal O\subset X\setminus\{0\}
\]
be an admissible class invariant under this action:
\begin{equation}\label{eq:G-domain-invariance}
	\mathcal G_s(\mathcal O)\subset\mathcal O,
	\qquad s>0.
\end{equation}

For $u\in\mathcal O$, define the orbital Rayleigh profile
\begin{equation}\label{eq:orbital-Rayleigh-profile}
	r_u(s)
	:=
	\mathcal R(\mathcal G_su),
	\qquad s>0.
\end{equation}
The scalar equation
\begin{equation}\label{eq:orbital-level-equation}
	r_u(s)=\lambda
\end{equation}
describes the intersections of the transformation orbit with the Rayleigh
level $\mathcal R=\lambda$. Its critical points satisfy
\begin{equation}\label{eq:orbital-criticality}
	r_u'(s)=0.
\end{equation}
At this stage these are only critical points of the constraint geometry,
not necessarily bifurcation points of the original equation.

The decisive situation occurs when the Rayleigh representation is exact
along an orbit. We say that a profile $\phi\in\mathcal O$ generates an
\emph{exact Rayleigh reduction} if
\begin{equation}\label{eq:exact-Rayleigh-orbit}
	F\bigl(
	\mathcal R(\mathcal G_s\phi),
	\mathcal G_s\phi
	\bigr)
	=
	0,
	\qquad s>0.
\end{equation}
Then
\begin{equation}\label{eq:exact-Rayleigh-profile}
	\lambda_\phi(s)
	:=
	\mathcal R(\mathcal G_s\phi)
\end{equation}
is the actual parameter law of the solution branch
\[
s
\longmapsto
\bigl(
\lambda_\phi(s),\mathcal G_s\phi
\bigr).
\]
Thus the geometry of the nonlinear solution set along this branch is
reduced exactly to the geometry of a one-dimensional Rayleigh profile.

This is the basic Rayleigh reduction principle: once an exact profile
$\phi$ has been identified, turning bifurcations can be detected directly
from the scalar condition
\begin{equation}\label{eq:exact-Rayleigh-critical}
	\boxed{
		\lambda_\phi'(s_*)=0.
	}
\end{equation}

\subsection{Rayleigh bifurcation criterion}
\label{subsec:Rayleigh-bifurcation-criterion}

The following result makes the preceding geometric mechanism precise.

\begin{theorem}[Rayleigh bifurcation criterion]
	\label{thm:Rayleigh-bifurcation}
	
	Let $\phi\in\mathcal O$ generate an exact Rayleigh branch
	\[
	F\bigl(
	\lambda_\phi(s),
	\mathcal G_s\phi
	\bigr)=0,
	\qquad
	\lambda_\phi(s)
	=
	\mathcal R(\mathcal G_s\phi).
	\]
	Assume that $\lambda_\phi$ is of class $C^2$ near some $s_*>0$ and
	\[
	\lambda_\phi'(s_*)=0,
	\qquad
	\lambda_\phi''(s_*)\neq0.
	\]
	Assume also that the orbit is of class $C^1$ and regular at $s_*$:
	\[
	\left.
	\frac{d}{ds}\mathcal G_s\phi
	\right|_{s=s_*}
	\neq0.
	\]
	
	Then
	\[
	(\lambda_\phi^*,u_*),
	\qquad
	\lambda_\phi^*:=\lambda_\phi(s_*),
	\qquad
	u_*:=\mathcal G_{s_*}\phi,
	\]
	is a nondegenerate turning point of the exact Rayleigh branch; in the
	terminology adopted above, it is a turning bifurcation of that branch.
	More precisely:
	
	\begin{itemize}
		\item if $\lambda_\phi''(s_*)<0$, then there exists
		$\varepsilon>0$ such that for every
		\[
		\lambda\in
		(\lambda_\phi^*-\varepsilon,\lambda_\phi^*)
		\]
		there are two distinct exact solutions
		\[
		u_-(\lambda),u_+(\lambda),
		\]
		with
		\[
		F(\lambda,u_\pm(\lambda))=0,
		\qquad
		u_\pm(\lambda)\to u_*
		\quad\text{as }\lambda\uparrow\lambda_\phi^*;
		\]
		
		\item if $\lambda_\phi''(s_*)>0$, then there exists
		$\varepsilon>0$ such that for every
		\[
		\lambda\in
		(\lambda_\phi^*,\lambda_\phi^*+\varepsilon)
		\]
		there are two distinct exact solutions
		\[
		u_-(\lambda),u_+(\lambda),
		\]
		with
		\[
		F(\lambda,u_\pm(\lambda))=0,
		\qquad
		u_\pm(\lambda)\to u_*
		\quad\text{as }\lambda\downarrow\lambda_\phi^*.
		\]
	\end{itemize}
	
\end{theorem}

\begin{proof}
	We consider the case
	\[
	\lambda_\phi''(s_*)<0;
	\]
	the case $\lambda_\phi''(s_*)>0$ is completely analogous, with the
	inequalities reversed.
	
	Since
	\[
	\lambda_\phi'(s_*)=0,
	\qquad
	\lambda_\phi''(s_*)<0,
	\]
	the point $s_*$ is a strict nondegenerate local maximum of the scalar
	profile. Hence, for $\lambda<\lambda_\phi^*$ sufficiently close to
	$\lambda_\phi^*$, there exist
	\[
	s_-(\lambda)<s_*<s_+(\lambda)
	\]
	such that
	\[
	\lambda_\phi(s_\pm(\lambda))=\lambda.
	\]
	By exactness,
	\[
	u_\pm(\lambda)
	:=
	\mathcal G_{s_\pm(\lambda)}\phi
	\]
	are solutions of $F(\lambda,u)=0$.
	
	The regularity condition
	\[
	\left.
	\frac{d}{ds}\mathcal G_s\phi
	\right|_{s=s_*}
	\neq0
	\]
	also implies local injectivity of the orbit near $s_*$. Indeed, by the
	Hahn--Banach theorem one can choose $\ell\in X^*$ such that
	$\ell(u'(s_*))\neq0$, where $u(s)=\mathcal G_s\phi$; hence
	$\ell(u(s))$ is strictly monotone in a sufficiently small neighborhood
	of $s_*$. Therefore $s_-(\lambda)\neq s_+(\lambda)$ gives
	\[
	u_-(\lambda)\neq u_+(\lambda).
	\]
	Finally,
	\[
	s_\pm(\lambda)\to s_*
	\]
	as $\lambda\uparrow\lambda_\phi^*$, and continuity of the orbit yields
	\[
	u_\pm(\lambda)\to\mathcal G_{s_*}\phi=u_*.
	\]
	
	For $\lambda_\phi''(s_*)>0$, the point $s_*$ is instead a strict
	nondegenerate local minimum, and the same argument gives two exact
	solutions for $\lambda>\lambda_\phi^*$ sufficiently close to
	$\lambda_\phi^*$.
\end{proof}
The importance of this criterion is that the turning point is obtained
from the explicit Rayleigh geometry of the exact branch itself. Neither a
singular point nor a kernel of the linearized operator needs to be known in
advance.

\subsection{Class-dependent Rayleigh thresholds}
\label{subsec:class-Rayleigh-thresholds}

The exact reduction may also be combined with a global Rayleigh threshold.
Let
\[
\mathcal S\subseteq\mathcal O
\]
be an admissible class, for example a positive, nonnegative, radial, or
otherwise invariant class.

For $u\in\mathcal S$, define the upper orbital level
\begin{equation}\label{eq:S-orbit-max}
	m_{\mathcal G}(u)
	:=
	\sup_{s>0}
	\mathcal R(\mathcal G_su),
\end{equation}
and the corresponding class-dependent Rayleigh threshold
\begin{equation}\label{eq:G-S-second-threshold}
	\boxed{
		\lambda_{\mathcal G,\mathcal S}^{**}
		:=
		\sup_{u\in\mathcal S}
		\sup_{s>0}
		\mathcal R(\mathcal G_su).
	}
\end{equation}

If $\mathcal S$ is invariant under the transformation,
\[
\mathcal G_s(\mathcal S)\subseteq\mathcal S,
\qquad s>0,
\]
then
\begin{equation}\label{eq:G-S-second-threshold-sup}
	\boxed{
		\lambda_{\mathcal G,\mathcal S}^{**}
		=
		\sup_{u\in\mathcal S}\mathcal R(u).
	}
\end{equation}
Indeed, $\mathcal G_1=\mathrm{Id}$ gives
\[
\mathcal R(u)
\le
\sup_{s>0}\mathcal R(\mathcal G_su),
\]
whereas invariance implies
\[
\sup_{s>0}\mathcal R(\mathcal G_su)
\le
\sup_{w\in\mathcal S}\mathcal R(w).
\]

Consequently,
\begin{equation}\label{eq:S-no-solutions-above}
	\boxed{
		F(\lambda,u)=0,\quad u\in\mathcal S
		\qquad\Longrightarrow\qquad
		\lambda\le
		\lambda_{\mathcal G,\mathcal S}^{**}.
	}
\end{equation}
Thus $\lambda_{\mathcal G,\mathcal S}^{**}$ is the maximal Rayleigh level
compatible with solutions in $\mathcal S$.

If an exact branch generated by $\phi\in\mathcal S$ satisfies
\begin{equation}\label{eq:phi-global-maximum}
	\lambda_\phi(s_*)
	=
	\max_{s>0}\lambda_\phi(s)
	=
	\lambda_{\mathcal G,\mathcal S}^{**},
\end{equation}
then the turning bifurcation point detected by
Theorem~\ref{thm:Rayleigh-bifurcation} occurs at the maximal Rayleigh
threshold of the class. In this case the same Rayleigh construction
identifies both the local turning mechanism and the global upper parameter
boundary.

\begin{remark}[Lower Rayleigh thresholds]
	The preceding construction is formulated for upper Rayleigh thresholds
	and therefore corresponds naturally to exact branches whose Rayleigh
	profiles attain a maximum. If instead
	\[
	\lambda_\phi''(s_*)>0,
	\]
	Theorem~\ref{thm:Rayleigh-bifurcation} detects a turning bifurcation at
	a local minimum. An analogous lower-threshold construction is obtained
	by replacing $\sup$ and $\max$ with $\inf$ and $\min$ and reversing the
	corresponding inequalities. Since the applications below concern upper
	existence thresholds, we use only the upper formulation.
\end{remark}

\begin{remark}[Relation with classical bifurcation methods]
	\label{rem:classical-bifurcation-methods}
	
	The exact Rayleigh reduction is complementary to classical local
	bifurcation theory. In implicit-function and Lyapunov--Schmidt
	approaches one usually starts from a known singular solution
	\[
	F(\lambda_*,u_*)=0,
	\qquad
	\ker F_u(\lambda_*,u_*)\neq\{0\},
	\]
	and then studies the local solution set near $(\lambda_*,u_*)$;
	see, for example,
	\cite{CrandallRabinowitz1971, Kielhofer2012}.
	Global and topological methods study continua and branching of solutions;
	see \cite{KrasnoselskiiZabreiko1984, Rabinowitz1971}, while numerical
	continuation detects singular points along computed solution branches
	\cite{DoedelKellerKernevez1991I,DoedelKellerKernevez1991II}.
	
	The exact Rayleigh reduction reverses this order when an appropriate
	transformation is available. It first constructs an exact solution branch
	and detects its critical parameter from
	\[
	\lambda_\phi'(s_*)=0.
	\]
	The resulting point may then be analyzed by classical local, global, or
	numerical bifurcation methods.
	
\end{remark}

\subsection{Connection with the classical fold criterion}
\label{subsec:Rayleigh-classical-fold}

The Rayleigh bifurcation criterion is geometric and does not require a
Fredholm decomposition or a Lyapunov--Schmidt reduction. Nevertheless, the
geometry of an exact Rayleigh branch naturally contains the classical fold
conditions when the required operator assumptions are available.

Assume that $F$ is of class $C^2$ and let
\[
u(s)=\mathcal G_s\phi,
\qquad
\lambda_\phi(s)=\mathcal R(\mathcal G_s\phi)
\]
be an exact branch:
\[
F(\lambda_\phi(s),u(s))=0.
\]
Differentiating with respect to $s$ gives
\begin{equation}\label{eq:exact-branch-first-derivative}
	F_\lambda(\lambda_\phi(s),u(s))\lambda_\phi'(s)
	+
	F_u(\lambda_\phi(s),u(s))u'(s)
	=
	0.
\end{equation}
At a critical point $s_*$,
\[
\lambda_\phi'(s_*)=0,
\]
and therefore
\begin{equation}\label{eq:Rayleigh-kernel-direction}
	\boxed{
		F_u(\lambda_\phi^*,u_*)u'(s_*)=0,
	}
\end{equation}
where
\[
u_*=\mathcal G_{s_*}\phi,
\qquad
\lambda_\phi^*=\lambda_\phi(s_*).
\]
Thus the tangent direction to the exact Rayleigh orbit automatically belongs
to the kernel of the state linearization. The kernel direction is therefore
a consequence of the Rayleigh reduction rather than an assumption used to
locate the critical point.

Differentiating the exact branch identity once more and using
$\lambda_\phi'(s_*)=0$ yields
\[
F_u(\lambda_\phi^*,u_*)u''(s_*)
+
D_{uu}F(\lambda_\phi^*,u_*)
[u'(s_*),u'(s_*)]
+
\lambda_\phi''(s_*)
F_\lambda(\lambda_\phi^*,u_*)
=
0.
\]
Let
\[
0\neq\psi
\in
\ker F_u(\lambda_\phi^*,u_*)^*.
\]
Pairing with $\psi$ eliminates the first term and gives
\begin{equation}\label{eq:Rayleigh-LS-identity}
	\boxed{
		\left\langle
		\psi,
		D_{uu}F(\lambda_\phi^*,u_*)
		[u'(s_*),u'(s_*)]
		\right\rangle
		=
		-
		\lambda_\phi''(s_*)
		\left\langle
		\psi,
		F_\lambda(\lambda_\phi^*,u_*)
		\right\rangle.
	}
\end{equation}

Suppose, in addition, that
\[
\ker F_u(\lambda_\phi^*,u_*)
=
\operatorname{span}\{u'(s_*)\},
\]
that $F_u(\lambda_\phi^*,u_*)$ is Fredholm of index zero, and that the
parameter transversality condition
\[
\left\langle
\psi,
F_\lambda(\lambda_\phi^*,u_*)
\right\rangle
\neq0
\]
holds. If
\[
\lambda_\phi''(s_*)\neq0,
\]
then \eqref{eq:Rayleigh-LS-identity} shows that
\[
\left\langle
\psi,
D_{uu}F(\lambda_\phi^*,u_*)
[u'(s_*),u'(s_*)]
\right\rangle
\neq0.
\]
Hence the quadratic Lyapunov--Schmidt coefficient is nonzero, and
$(\lambda_\phi^*,u_*)$ is a classical simple fold.

Thus exact Rayleigh reduction and classical fold theory provide two
different levels of information. The Rayleigh reduction constructs the
branch and detects its critical parameter; the classical fold conditions,
when available, identify the local singularity type of the detected point.

\section{Amplitude Rayleigh geometry on bounded domains}
\label{sec:Bounded-exact-Rayleigh}

We first illustrate the exact Rayleigh reduction by the weighted Kirchhoff
problem
\begin{equation}\label{eq:bounded-K}
	\begin{cases}
		-\left(
		a+\lambda\displaystyle\int_\Omega |\nabla u|^2\,dx
		\right)\Delta u
		=
		k(x)|u|^{q-2}u,
		&x\in\Omega,
		\\[2mm]
		u=0,
		&x\in\partial\Omega,
	\end{cases}
\end{equation}
where $\Omega\subset\mathbb R^N$ is a bounded smooth domain,
\[
a>0,
\qquad
k\in C(\overline\Omega),
\qquad
k(x)>0
\quad\text{on }\overline\Omega,
\]
and
\[
2<q<\min\{4,2^*\},
\qquad
2^*=
\begin{cases}
	\dfrac{2N}{N-2}, & N\ge3,\\
	+\infty, & N\le2.
\end{cases}
\]

Set
\[
A(u):=\int_\Omega|\nabla u|^2\,dx,
\qquad
B(u):=\int_\Omega k(x)|u|^q\,dx,
\]
and write
\[
F(\lambda,u)
:=
-\bigl(a+\lambda A(u)\bigr)\Delta u
-k(x)|u|^{q-2}u.
\]
Then \eqref{eq:bounded-K} is equivalent to
\[
F(\lambda,u)=0,
\qquad
u\in H_0^1(\Omega).
\]

Testing \eqref{eq:bounded-K} by $u\neq0$ gives
\[
\bigl(a+\lambda A(u)\bigr)A(u)=B(u),
\]
and hence every nontrivial solution satisfies the Nehari Rayleigh
representation
\begin{equation}\label{eq:parameter-Rayleigh-K}
	\boxed{
		\lambda
		=
		\mathcal R_{\mathcal N}(u)
		:=
		\frac{B(u)-aA(u)}{A(u)^2}.
	}
\end{equation}
We now show that, along distinguished amplitude directions, this scalar
parameter representation becomes the exact parameter law of a genuine
solution branch.

\subsection{Exact amplitude branches and turning bifurcations}
\label{subsec:bounded-explicit-turn}

Consider the amplitude action
\[
\mathcal G_s w=sw,
\qquad s>0,
\]
and define the orbital Rayleigh profile
\[
\lambda_w(s)
:=
\mathcal R_{\mathcal N}(sw).
\]
Since
\[
A(sw)=s^2A(w),
\qquad
B(sw)=s^qB(w),
\]
we obtain
\begin{equation}\label{eq:bounded-Rayleigh-general}
	\lambda_w(s)
	=
	\frac{B(w)}{A(w)^2}s^{q-4}
	-
	\frac{a}{A(w)}s^{-2}.
\end{equation}

Set
\begin{equation}\label{eq:tau-w}
	\tau_w:=\frac{B(w)}{A(w)}.
\end{equation}
Since $k>0$, we have $\tau_w>0$ for every $w\neq0$. Moreover,
\[
a+\lambda_w(s)A(sw)
=
\tau_ws^{q-2}.
\]
Substituting
\[
u=sw,
\qquad
\lambda=\lambda_w(s)
\]
into \eqref{eq:bounded-K}, we find that
\[
F(\lambda_w(s),sw)=0
\qquad\text{for all }s>0
\]
if and only if
\begin{equation}\label{eq:bounded-exact-profile}
	\boxed{
		\begin{cases}
			-\tau_w\Delta w
			=
			k(x)|w|^{q-2}w,
			&x\in\Omega,
			\\[1mm]
			w=0,
			&x\in\partial\Omega.
		\end{cases}
	}
\end{equation}

Thus \eqref{eq:bounded-exact-profile} is precisely the exactness condition
selecting those amplitude directions along which the Nehari Rayleigh
representation reconstructs an exact solution branch.

Writing
\[
w=\tau_w^{1/(q-2)}\phi,
\]
the exactness condition reduces to the normalized profile equation
\begin{equation}\label{eq:normalized-bounded}
	\boxed{
		\begin{cases}
			-\Delta\phi
			=
			k(x)|\phi|^{q-2}\phi,
			&x\in\Omega,
			\\[1mm]
			\phi=0,
			&x\in\partial\Omega.
		\end{cases}
	}
\end{equation}

Let $\phi\neq0$ solve \eqref{eq:normalized-bounded}. Testing by $\phi$
gives
\begin{equation}\label{eq:normalized-Nehari}
	A(\phi)=B(\phi)=:K_\phi.
\end{equation}
Hence the amplitude orbit
\[
\mathcal G_s\phi=s\phi,
\qquad s>0,
\]
generates the exact solution branch
\begin{equation}\label{eq:Gamma-phi}
	\Gamma_\phi
	=
	\left\{
	(\lambda_\phi(s),s\phi):s>0
	\right\},
\end{equation}
where
\begin{equation}\label{eq:bounded-branch}
	\boxed{
		\lambda_\phi(s)
		=
		\mathcal R_{\mathcal N}(s\phi)
		=
		\frac{s^{q-2}-a}{s^2K_\phi}.
	}
\end{equation}
In particular,
\[
F\bigl(\lambda_\phi(s),s\phi\bigr)=0,
\qquad s>0.
\]

The turning geometry is completely explicit. The equation
\[
\lambda_\phi'(s)=0
\]
has the unique solution
\begin{equation}\label{eq:bounded-sstar}
	\boxed{
		s_*^{q-2}
		=
		\frac{2a}{4-q}.
	}
\end{equation}
Thus the critical amplitude $s_*$ is independent not only of the
normalized profile $\phi$, but also of the weight $k(x)$. The spatial
heterogeneity affects the turning value only through $K_\phi$. Indeed,
\begin{equation}\label{eq:bounded-lambda-star}
	\boxed{
		\lambda_\phi^*
		:=
		\lambda_\phi(s_*)
		=
		\frac{q-2}{2K_\phi}s_*^{q-4}.
	}
\end{equation}
Moreover,
\begin{equation}\label{eq:bounded-second}
	\lambda_\phi''(s_*)
	=
	-\frac{(q-2)(4-q)}{K_\phi}
	s_*^{q-6}
	<0,
\end{equation}
so $s_*$ is a nondegenerate global maximum of the exact Rayleigh profile.

The remaining branch geometry follows directly from
\eqref{eq:bounded-branch}:
\[
\lambda_\phi(s)\to-\infty
\quad\text{as }s\downarrow0,
\qquad
\lambda_\phi(s)\to0^+
\quad\text{as }s\to\infty.
\]
The branch crosses $\lambda=0$ at the profile-independent amplitude
\begin{equation}\label{eq:bounded-zero-amplitude}
	s_{\rm z}:=a^{1/(q-2)},
\end{equation}
and
\[
s_{\rm z}<s_*.
\]
Consequently, $\Gamma_\phi$ contains exactly one solution for
$\lambda\le0$, exactly two solutions for
$0<\lambda<\lambda_\phi^*$, one turning solution for
$\lambda=\lambda_\phi^*$, and no solution for
$\lambda>\lambda_\phi^*$.

By Theorem~\ref{thm:Rayleigh-bifurcation},
\[
\boxed{
	(\lambda_\phi^*,s_*\phi)
}
\]
is a nondegenerate turning bifurcation point of the exact amplitude branch.

\begin{theorem}[Exact amplitude bifurcation branch]
	\label{thm:single-branch-geometry}
	
	Let $\phi\neq0$ solve the weighted normalized profile problem
	\eqref{eq:normalized-bounded}. Then
	\[
	\Gamma_\phi
	=
	\left\{
	(\lambda_\phi(s),s\phi):s>0
	\right\}
	\]
	is an exact solution branch of \eqref{eq:bounded-K}. It has a unique
	nondegenerate turning bifurcation point
	\[
	\boxed{
		(\lambda_\phi^*,s_*\phi),
		\qquad
		s_*
		=
		\left(
		\frac{2a}{4-q}
		\right)^{\frac1{q-2}},
	}
	\]
	where
	\[
	\lambda_\phi^*
	=
	\frac{q-2}{2K_\phi}s_*^{q-4},
	\qquad
	K_\phi
	=
	\int_\Omega|\nabla\phi|^2\,dx
	=
	\int_\Omega k(x)|\phi|^q\,dx.
	\]
	The branch contains exactly one solution for $\lambda\le0$, exactly
	two solutions for $0<\lambda<\lambda_\phi^*$, one turning solution
	for $\lambda=\lambda_\phi^*$, and no solution for
	$\lambda>\lambda_\phi^*$.
	
	In particular, the critical amplitude $s_*$ is universal: it is
	independent of both the normalized profile and the spatial weight
	$k(x)$.
\end{theorem}

\subsection{The maximal Rayleigh branch}
\label{subsec:bounded-nehari-extremal}

We next identify which normalized profile generates the largest turning
value. Let
\[
\mathcal S=H_0^1(\Omega)\setminus\{0\}.
\]
The upper Rayleigh threshold of the amplitude class is
\begin{equation}\label{eq:bounded-class-threshold}
	\lambda_{\mathcal G,\mathcal S}^{**}
	:=
	\sup_{w\in H_0^1(\Omega)\setminus\{0\}}
	\sup_{s>0}
	\mathcal R_{\mathcal N}(sw).
\end{equation}

For fixed $w\neq0$, the orbital profile
\eqref{eq:bounded-Rayleigh-general} has the unique critical point
\begin{equation}\label{eq:bounded-s-star-general}
	s_w^{q-2}
	=
	\frac{2aA(w)}
	{(4-q)B(w)}.
\end{equation}
Since
\[
\lambda_w(s)\to-\infty
\quad\text{as }s\downarrow0,
\qquad
\lambda_w(s)\to0^+
\quad\text{as }s\to\infty,
\]
this point is its unique global maximum. Substitution into
\eqref{eq:bounded-Rayleigh-general} gives
\begin{equation}\label{eq:bounded-profile-maximum}
	\boxed{
		\max_{s>0}\lambda_w(s)
		=
		C_{a,q}
		\left(
		\frac{B(w)^{2/q}}{A(w)}
		\right)^{\frac{q}{q-2}},
	}
\end{equation}
where
\begin{equation}\label{eq:Caq}
	C_{a,q}
	:=
	\frac{q-2}{2}
	\left(
	\frac{2a}{4-q}
	\right)^{\frac{q-4}{q-2}}.
\end{equation}

Introduce the optimal weighted Sobolev quotient
\begin{equation}\label{eq:bounded-Sobolev-constant}
	S_{q,k}
	:=
	\inf_{w\in H_0^1(\Omega)\setminus\{0\}}
	\frac{A(w)}
	{B(w)^{2/q}}.
\end{equation}
Then \eqref{eq:bounded-profile-maximum} yields
\begin{equation}\label{eq:bounded-threshold-explicit}
	\boxed{
		\lambda_{\mathcal G,\mathcal S}^{**}
		=
		C_{a,q}
		S_{q,k}^{-\frac{q}{q-2}}.
	}
\end{equation}

Since $q<2^*$ and $\Omega$ is bounded, the compact embedding
\[
H_0^1(\Omega)\hookrightarrow L^q(\Omega)
\]
and the boundedness of $k$ imply that the infimum in
\eqref{eq:bounded-Sobolev-constant} is attained. Let $w_0$ be a minimizer
normalized by
\[
B(w_0)=1.
\]
Then
\[
A(w_0)=S_{q,k},
\]
and its Euler--Lagrange equation is
\begin{equation}\label{eq:w0-Euler-Lagrange}
	-\Delta w_0
	=
	S_{q,k}\,
	k(x)|w_0|^{q-2}w_0.
\end{equation}

Define
\begin{equation}\label{eq:phi0-normalization}
	\phi_0
	:=
	S_{q,k}^{1/(q-2)}w_0.
\end{equation}
Then $\phi_0$ solves the weighted normalized profile equation
\eqref{eq:normalized-bounded}. Moreover,
\begin{equation}\label{eq:Kphi0}
	K_{\phi_0}
	=
	A(\phi_0)
	=
	B(\phi_0)
	=
	S_{q,k}^{\frac{q}{q-2}}.
\end{equation}
Consequently,
\[
\lambda_{\phi_0}^*
=
\frac{q-2}{2}
s_*^{q-4}
S_{q,k}^{-\frac{q}{q-2}}
=
C_{a,q}S_{q,k}^{-\frac{q}{q-2}},
\]
and comparison with \eqref{eq:bounded-threshold-explicit} gives
\begin{equation}\label{eq:bounded-threshold-attained}
	\boxed{
		\lambda_{\phi_0}^*
		=
		\lambda_{\mathcal G,\mathcal S}^{**}.
	}
\end{equation}

Moreover, $\phi_0$ is a least-energy nontrivial solution of the weighted
normalized profile equation. Indeed, if $\phi\neq0$ solves
\eqref{eq:normalized-bounded}, then
\[
A(\phi)=B(\phi),
\]
and hence, by the definition of $S_{q,k}$,
\[
S_{q,k}
\le
\frac{A(\phi)}{B(\phi)^{2/q}}
=
A(\phi)^{\frac{q-2}{q}}.
\]
Therefore
\[
A(\phi)
\ge
S_{q,k}^{\frac{q}{q-2}}
=
A(\phi_0).
\]
Since
\[
J_{\Omega,k}(\phi)
=
\frac12A(\phi)-\frac1qB(\phi)
=
\frac{q-2}{2q}A(\phi)
\]
on every normalized profile, $\phi_0$ has the least nontrivial critical
energy.

Thus the weighted ground-state profile $\phi_0$ generates the maximal exact
amplitude branch, and
\begin{equation}\label{eq:maximal-bifurcation-point}
	\boxed{
		\left(
		\lambda_{\mathcal G,\mathcal S}^{**},
		s_*\phi_0
		\right)
	}
\end{equation}
is its nondegenerate turning bifurcation point.

Moreover, this threshold is global. Indeed, every nontrivial solution of
\eqref{eq:bounded-K} satisfies
\[
\lambda
=
\mathcal R_{\mathcal N}(u)
\le
\lambda_{\mathcal G,\mathcal S}^{**}.
\]
Hence no nontrivial solution exists for
\[
\lambda>\lambda_{\mathcal G,\mathcal S}^{**}.
\]

\begin{corollary}[Maximal amplitude bifurcation branch]
	\label{cor:bounded-maximal-branch}
	
	The weighted ground-state profile $\phi_0$ generates the maximal exact
	amplitude branch, and its turning value satisfies
	\[
	\boxed{
		\lambda_{\phi_0}^*
		=
		\lambda_{\mathcal G,\mathcal S}^{**}
		=
		C_{a,q}S_{q,k}^{-\frac{q}{q-2}}.
	}
	\]
	In particular,
	\[
	\boxed{
		\left(
		\lambda_{\mathcal G,\mathcal S}^{**},
		s_*\phi_0
		\right)
	}
	\]
	is a nondegenerate turning bifurcation point at the global upper
	existence threshold for nontrivial solutions.
\end{corollary}

\subsection{Variational hierarchy of amplitude bifurcations}
\label{subsec:countable-Rayleigh-branches}

The exact amplitude reduction is not restricted to the weighted
ground-state profile $\phi_0$. Every nontrivial solution of
\eqref{eq:normalized-bounded} generates its own exact Rayleigh branch.
The variational multiplicity of the weighted normalized problem therefore
produces a countable hierarchy of amplitude bifurcations for the Kirchhoff
equation.

Consider the even functional
\[
J(u)
=
\frac12A(u)-\frac1qB(u),
\qquad
u\in H_0^1(\Omega).
\]
Its nontrivial critical points are precisely the weak solutions of
\eqref{eq:normalized-bounded}. Since $2<q<2^*$, the compact embedding
\[
H_0^1(\Omega)\hookrightarrow L^q(\Omega)
\]
and the boundedness of $k$ imply the Palais--Smale compactness condition
for $J$. Moreover, since
$k\in C(\overline\Omega)$ and $k>0$ on $\overline\Omega$, there exists
$k_0>0$ such that
\[
k(x)\ge k_0
\qquad\text{on }\overline\Omega,
\]
which yields the standard symmetric superlinear geometry. Hence the
symmetric critical point theorem gives an unbounded sequence of critical
values; see, for example,
\cite[Theorem~9.38]{Rabinowitz1986}.

As shown above, $\phi_0$ is a least-energy nontrivial critical point.
We may therefore choose $\phi_1=\phi_0$ and select the remaining profiles
from an unbounded variational sequence. Accordingly, there exists a
sequence of pairwise distinct antipodal solution pairs
\[
\{\pm\phi_j\}_{j\ge1}
\]
such that
\[
\phi_1=\phi_0,
\qquad
c_j:=J(\phi_j),
\qquad
0<c_1\le c_2\le\cdots,
\qquad
c_j\to+\infty.
\]

Testing \eqref{eq:normalized-bounded} by $\phi_j$ gives
\[
A(\phi_j)
=
B(\phi_j)
=
:K_j.
\]
Therefore
\begin{equation}\label{eq:cj-Kj}
	c_j
	=
	\frac{q-2}{2q}K_j,
	\qquad
	K_j
	=
	\frac{2q}{q-2}c_j
	\longrightarrow+\infty.
\end{equation}

Each antipodal profile pair $\{\pm\phi_j\}$ generates the pair of exact
amplitude branches
\[
\Gamma_j^\pm
=
\left\{
(\lambda_j(s),\pm s\phi_j):s>0
\right\},
\]
where
\begin{equation}\label{eq:lambda-j}
	\lambda_j(s)
	=
	\frac{s^{q-2}-a}{s^2K_j}.
\end{equation}
Since both the normalized equation and the Kirchhoff equation are odd in
$u$, the two branches have the same Rayleigh profile $\lambda_j(s)$.
In what follows, solutions are counted in antipodal pairs.

All these branch pairs have the same critical amplitude
\[
s_*^{q-2}
=
\frac{2a}{4-q},
\]
independent of both the profile and the weight $k(x)$. Their turning values
depend on the profiles only through $K_j$:
\begin{equation}\label{eq:lambda-j-star}
	\boxed{
		\lambda_j^*
		=
		\frac{q-2}{2K_j}s_*^{q-4}
		=
		\frac{(q-2)^2}{4q}
		s_*^{q-4}\frac1{c_j}.
	}
\end{equation}
Since $c_j$ is nondecreasing and tends to infinity,
\begin{equation}\label{eq:lambda-j-order}
	\boxed{
		\lambda_1^*
		=
		\lambda_{\mathcal G,\mathcal S}^{**}
		\ge
		\lambda_2^*
		\ge
		\cdots
		\longrightarrow0.
	}
\end{equation}

Thus the global upper Rayleigh threshold is the first element of a
countable hierarchy of turning bifurcation values accumulating at
$\lambda=0$.

\begin{theorem}[Countable hierarchy of amplitude bifurcations]
	\label{thm:infinite-scaling-hierarchy}
	
	Problem \eqref{eq:bounded-K} possesses infinitely many distinct
	antipodal pairs of exact amplitude branches
	\[
	\Gamma_j^\pm,
	\qquad
	j\ge1.
	\]
	The branches $\Gamma_j^\pm$ have the antipodal nondegenerate turning
	bifurcation points
	\[
	(\lambda_j^*,\pm s_*\phi_j),
	\]
	and the corresponding bifurcation values satisfy
	\[
	\boxed{
		\lambda_{\mathcal G,\mathcal S}^{**}
		=
		\lambda_1^*
		\ge
		\lambda_2^*
		\ge
		\cdots
		\longrightarrow0.
	}
	\]
	In particular, $\lambda_1^*$ is the global upper existence threshold
	for nontrivial solutions.
	
	More precisely:
	\begin{enumerate}[label=\textup{(\roman*)}]
		\item for every $\lambda\le0$, problem
		\eqref{eq:bounded-K} possesses infinitely many distinct antipodal
		pairs of nontrivial solutions;
		
		\item for every $m\in\mathbb N$ and
		\[
		0<\lambda<\lambda_m^*,
		\]
		problem \eqref{eq:bounded-K} possesses at least $2m$ distinct
		antipodal pairs of nontrivial solutions;
		
		\item for
		\[
		\lambda>
		\lambda_1^*
		=
		\lambda_{\mathcal G,\mathcal S}^{**},
		\]
		problem \eqref{eq:bounded-K} has no nontrivial solutions.
	\end{enumerate}
\end{theorem}

\begin{proof}
	We first show that distinct antipodal profile pairs generate distinct
	antipodal branch pairs. Suppose that, for some $i\neq j$ and $s,r>0$,
	\[
	s\phi_i=\pm r\phi_j.
	\]
	Then
	\[
	\phi_i=\pm\alpha\phi_j,
	\qquad
	\alpha:=\frac{r}{s}>0.
	\]
	Since both profiles solve \eqref{eq:normalized-bounded}, substitution
	of $\phi_i=\pm\alpha\phi_j$ gives
	\[
	\alpha=\alpha^{q-1}.
	\]
	Because $q>2$, we obtain $\alpha=1$, and hence
	\[
	\{\pm\phi_i\}
	=
	\{\pm\phi_j\},
	\]
	a contradiction. Thus the branch pairs $\Gamma_j^\pm$ are mutually
	distinct.
	
	Formula \eqref{eq:lambda-j-star} and $c_j\to+\infty$ imply
	\[
	\lambda_j^*\to0,
	\]
	while the ordering of $c_j$ gives
	\[
	\lambda_1^*
	\ge
	\lambda_2^*
	\ge
	\cdots.
	\]
	Since $\phi_1=\phi_0$, equality
	\[
	\lambda_1^*
	=
	\lambda_{\mathcal G,\mathcal S}^{**}
	\]
	follows from \eqref{eq:bounded-threshold-attained}.
	
	For every $j$,
	\[
	\lambda_j''(s_*)<0,
	\qquad
	\left.
	\frac{d}{ds}\mathcal G_s\phi_j
	\right|_{s=s_*}
	=
	\phi_j\neq0.
	\]
	Hence Theorem~\ref{thm:Rayleigh-bifurcation} shows that
	\[
	(\lambda_j^*,s_*\phi_j)
	\]
	is a nondegenerate turning bifurcation point of $\Gamma_j^+$.
	By odd symmetry,
	\[
	(\lambda_j^*,-s_*\phi_j)
	\]
	is the corresponding turning bifurcation point of $\Gamma_j^-$.
	
	Let $\lambda\le0$. By
	Theorem~\ref{thm:single-branch-geometry}, for each $j$ there is exactly
	one amplitude $s>0$ at the level $\lambda$. Hence the branch pair
	$\Gamma_j^\pm$ yields exactly one antipodal pair
	\[
	\{\pm s\phi_j\}
	\]
	of nontrivial solutions. Since the branch pairs are mutually distinct,
	there are infinitely many such solution pairs. This proves
	\textup{(i)}.
	
	Now let $m\in\mathbb N$ and
	\[
	0<\lambda<\lambda_m^*.
	\]
	For every $j\le m$,
	\[
	\lambda
	<
	\lambda_m^*
	\le
	\lambda_j^*.
	\]
	Hence the scalar profile $\lambda_j(s)$ has two distinct amplitudes
	\[
	s_j^-(\lambda)<s_*<s_j^+(\lambda)
	\]
	at the level $\lambda$. These amplitudes generate the two distinct
	antipodal solution pairs
	\[
	\{\pm s_j^-(\lambda)\phi_j\},
	\qquad
	\{\pm s_j^+(\lambda)\phi_j\}.
	\]
	Thus each of the first $m$ profile pairs contributes two distinct
	antipodal pairs of solutions, and \eqref{eq:bounded-K} has at least
	$2m$ such pairs. This proves \textup{(ii)}.
	
	Finally, suppose
	\[
	\lambda>
	\lambda_1^*
	=
	\lambda_{\mathcal G,\mathcal S}^{**}.
	\]
	By the definition of the maximal Rayleigh threshold,
	\[
	\mathcal R_{\mathcal N}(u)
	\le
	\lambda_{\mathcal G,\mathcal S}^{**}
	<
	\lambda
	\qquad
	\text{for every }u\neq0.
	\]
	Every nontrivial solution of \eqref{eq:bounded-K}, however, must satisfy
	\[
	\lambda
	=
	\mathcal R_{\mathcal N}(u)
	=
	\frac{B(u)-aA(u)}{A(u)^2},
	\]
	which is impossible. This proves \textup{(iii)}.
\end{proof}

Thus the unbounded variational hierarchy of the weighted normalized
profile problem is transformed by the exact amplitude Rayleigh reduction
into a countable hierarchy of explicit turning bifurcation values,
\[
\lambda_j^*\longrightarrow0.
\]
The weighted ground-state branch determines the global upper threshold,
while higher variational profiles generate successively smaller
bifurcation levels. The spatial heterogeneity $k(x)$ changes the profile
energies and hence the individual thresholds $\lambda_j^*$, but leaves the
universal critical amplitude $s_*$ unchanged.

\begin{figure}[t]
	\centering
	\includegraphics[width=.82\textwidth]{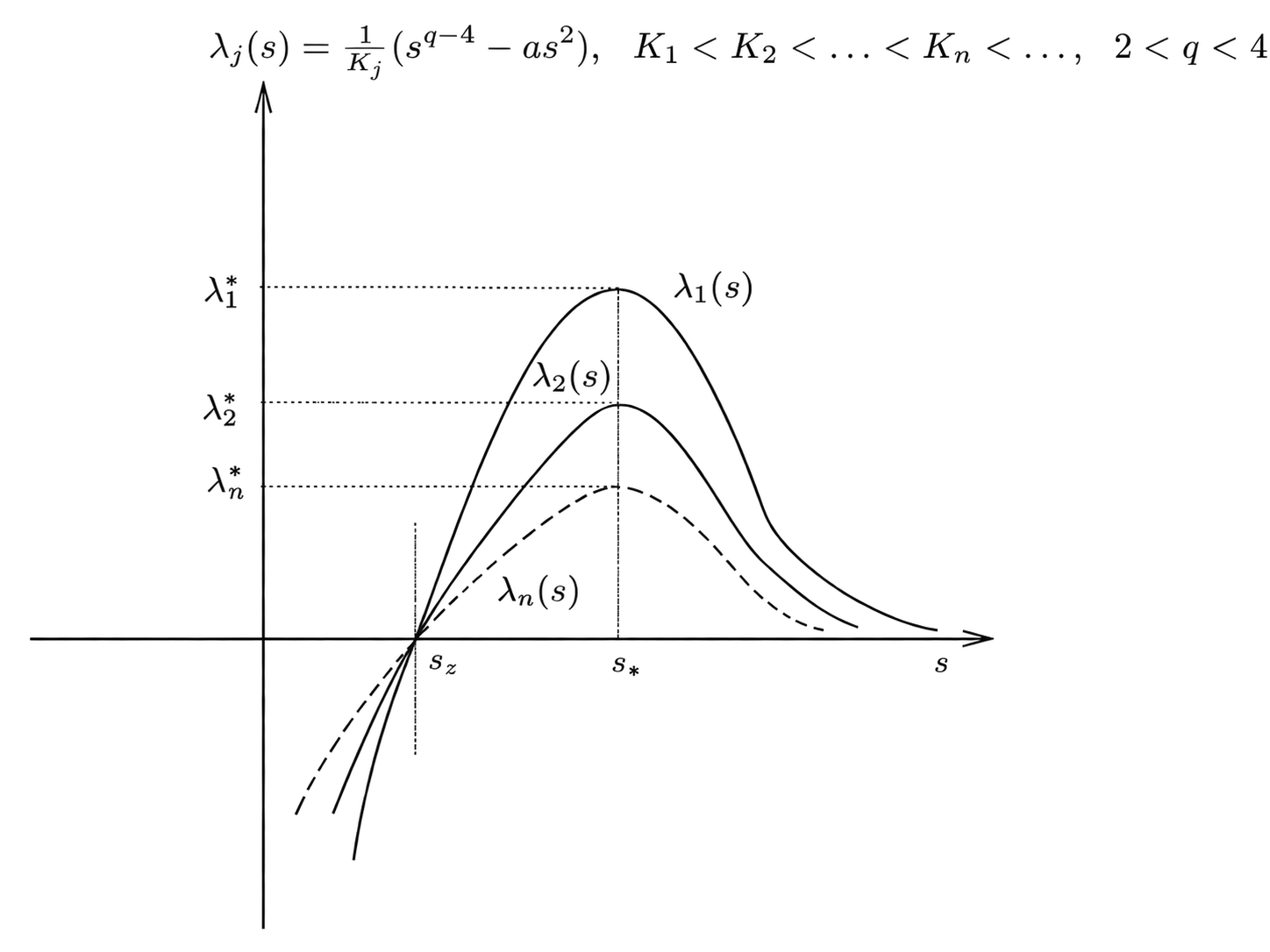}
	\caption{
		Exact amplitude Rayleigh profiles for representative normalized profiles.
		All branches have the same zero-crossing amplitude $s_{\rm z}$ and the
		same critical amplitude $s_*$. Their turning bifurcation values
		$\lambda_j^*$ are determined by the profile levels $K_j$ of the
		weighted normalized problem and accumulate at $\lambda=0$.
	}
	\label{fig:bounded-branch-geometry}
\end{figure}


\section{Exact spatial Rayleigh reduction for generalized Kirchhoff equations}
\label{sec:kirchhoff-exact-Rayleigh}

The variational theory of nonlinear scalar-field equations on
$\mathbb R^N$ goes back, in particular, to the pioneering work of
Strauss \cite{Strauss1977}. Under the nearly optimal hypotheses developed
by Berestycki and Lions
\cite{BerestyckiLions1983I,BerestyckiLions1983II}, the scalar-field equation
\[
-\Delta\phi=g(\phi)
\qquad\text{in }\mathbb R^N
\]
admits a positive radial ground state and an infinite sequence of distinct
radial bound states.

We consider the corresponding Kirchhoff generalization
\begin{equation}\label{eq:general-K}
	-\left(
	a+b
	\left(
	\int_{\mathbb R^N}|\nabla u|^2\,dx
	\right)^\theta
	\right)\Delta u
	=
	g(u)
	\qquad\text{in }\mathbb R^N,
\end{equation}
where
\[
N\ge3,
\qquad
a>0,\quad b\ge0,\quad \theta>0.
\]
The standard Kirchhoff law corresponds to $\theta=1$.

We assume that $g\in C(\mathbb R,\mathbb R)$ is odd and satisfies the
Berestycki--Lions hypotheses
\begin{align}
	-\infty
	&<
	\liminf_{s\to0^+}\frac{g(s)}{s}
	\le
	\limsup_{s\to0^+}\frac{g(s)}{s}
	=
	-m
	<0,
	\label{eq:BL-g1}
	\\
	-\infty
	&\le
	\limsup_{s\to+\infty}
	\frac{g(s)}{s^{(N+2)/(N-2)}}
	\le0,
	\label{eq:BL-g2}
\end{align}
and that there exists $\zeta>0$ such that
\begin{equation}\label{eq:BL-g3}
	G(\zeta)>0,
	\qquad
	G(t):=\int_0^t g(\tau)\,d\tau.
\end{equation}
The power model is recovered by taking
\[
g(t)=|t|^{p-2}t-\mu t,
\qquad
\mu>0,
\qquad
2<p<2^*:=\frac{2N}{N-2}.
\]
Thus the negative linear behavior of $g$ at the origin incorporates the
mass term of the usual scalar-field equation.

Set
\[
A(u)
:=
\int_{\mathbb R^N}|\nabla u|^2\,dx,
\qquad
\mathcal B(u)
:=
\int_{\mathbb R^N}G(u)\,dx,
\]
whenever the latter integral is finite, and define
\[
\mathcal D_G
:=
\left\{
u\in H^1(\mathbb R^N)\setminus\{0\}:
\mathcal B(u)\in\mathbb R
\right\}.
\]
The set $\mathcal D_G$ is invariant under spatial dilations.
Throughout this section, by a nontrivial solution we mean a finite-energy
solution $u\in\mathcal D_G$; whenever the Pohozaev identity is used, the
standard regularity required for its validity is understood.

Then \eqref{eq:general-K} takes the form
\[
-\bigl(a+bA(u)^\theta\bigr)\Delta u=g(u).
\]
We write it as
\[
F_\theta(b,u)=0,
\]
where
\begin{equation}\label{eq:general-F}
	F_\theta(b,u)
	:=
	-\bigl(a+bA(u)^\theta\bigr)\Delta u
	-g(u).
\end{equation}

The associated energy functional on $\mathcal D_G$ is
\begin{equation}\label{eq:general-energy}
	E_{b,\theta}(u)
	=
	\frac a2A(u)
	+
	\frac{b}{2(\theta+1)}A(u)^{\theta+1}
	-
	\mathcal B(u).
\end{equation}

The crucial feature of \eqref{eq:general-K} is that spatial dilation acts
on the nonlocal coefficient through a fixed homogeneity law, whereas no
homogeneity assumption on $g$ is required. This makes it possible to
separate the geometry of the exact Kirchhoff branches from the particular
form of the local nonlinearity.

\subsection{Pohozaev parameter representation}
\label{subsec:general-Pohozaev-functional}

We first derive the scalar parameter relation naturally associated with
spatial dilation. For a fixed solution $u$, the quantity
\[
\kappa_u:=a+bA(u)^\theta
\]
is a positive constant, and \eqref{eq:general-K} becomes
\[
-\kappa_u\Delta u=g(u).
\]
Hence every sufficiently regular nontrivial solution for which the
Pohozaev identity is valid satisfies
\begin{equation}\label{eq:general-Pohozaev-identity}
	\frac{N-2}{2}
	\bigl(
	aA(u)+bA(u)^{\theta+1}
	\bigr)
	-
	N\mathcal B(u)
	=
	0.
\end{equation}
Solving for $b$ gives the Pohozaev Rayleigh functional on
$\mathcal D_G$,
\begin{equation}\label{eq:general-Pohozaev-Rayleigh}
	\boxed{
		\mathcal R_{\mathcal P,\theta}(u)
		:=
		\frac{
			2N\mathcal B(u)
			-
			(N-2)aA(u)
		}{
			(N-2)A(u)^{\theta+1}
		}.
	}
\end{equation}
Thus every sufficiently regular nontrivial solution satisfies
\begin{equation}\label{eq:general-Pohozaev-parameter}
	\boxed{
		b
		=
		\mathcal R_{\mathcal P,\theta}(u).
	}
\end{equation}

As in the abstract setting, this identity alone is only a necessary scalar
constraint:
\[
b=\mathcal R_{\mathcal P,\theta}(u)
\qquad\not\Longrightarrow\qquad
F_\theta(b,u)=0.
\]
The exact spatial reduction below identifies distinguished dilation orbits
on which the Pohozaev Rayleigh functional becomes the actual parameter law
of an exact solution branch. Remarkably, the resulting branch geometry
depends on $N$ and $\theta$, but not on the particular form of the
Berestycki--Lions nonlinearity $g$.

\subsection{Exact spatial reduction and normalized profiles}
\label{subsec:general-exact-spatial}

We restrict the Pohozaev Rayleigh functional to a spatial dilation orbit
\[
(\mathcal G_s w)(x)=w(sx),
\qquad s>0,
\]
where
\[
w\in\mathcal D_G.
\]
Under spatial dilation,
\[
A(\mathcal G_s w)
=
s^{2-N}A(w),
\qquad
\mathcal B(\mathcal G_s w)
=
s^{-N}\mathcal B(w).
\]
Hence $\mathcal G_s w\in\mathcal D_G$ for every $s>0$, and
\[
\mathcal R_{\mathcal P,\theta}(\mathcal G_s w)
=
\frac{
	2Ns^{-N}\mathcal B(w)
	-
	(N-2)as^{2-N}A(w)
}{
	(N-2)
	\bigl(s^{2-N}A(w)\bigr)^{\theta+1}
}.
\]

Introduce
\begin{equation}\label{eq:general-sigma-w}
	\sigma_w
	:=
	\frac{2N\mathcal B(w)}{(N-2)A(w)}.
\end{equation}
Then the orbital Pohozaev Rayleigh profile becomes
\begin{equation}\label{eq:general-spatial-profile}
	\boxed{
		b_w(s)
		:=
		\mathcal R_{\mathcal P,\theta}(\mathcal G_s w)
		=
		\frac{s^{\theta(N-2)-2}}{A(w)^\theta}
		\left(
		\sigma_w-as^2
		\right).
	}
\end{equation}

At this stage \eqref{eq:general-spatial-profile} enforces only the
Pohozaev constraint. To determine when it generates an exact solution
branch, set
\[
u=\mathcal G_s w,
\qquad
b=b_w(s).
\]
Since
\[
A(\mathcal G_s w)^\theta
=
s^{-\theta(N-2)}A(w)^\theta,
\]
we obtain
\[
b_w(s)A(\mathcal G_s w)^\theta
=
s^{-2}
\left(
\sigma_w-as^2
\right),
\]
and therefore
\begin{equation}\label{eq:general-effective-coefficient}
	\boxed{
		a+
		b_w(s)A(\mathcal G_s w)^\theta
		=
		\sigma_w s^{-2}.
	}
\end{equation}
On the other hand,
\[
\Delta(\mathcal G_s w)(x)
=
s^2\Delta w(sx).
\]
Thus the factors $s^{-2}$ and $s^2$ cancel exactly, and substitution into
\eqref{eq:general-K} shows that
\[
F_\theta
\bigl(
b_w(s),\mathcal G_s w
\bigr)
=
0
\qquad\text{for all }s>0
\]
if and only if
\begin{equation}\label{eq:general-exactness}
	\boxed{
		-\sigma_w\Delta w
		=
		g(w)
		\qquad\text{in }\mathbb R^N.
	}
\end{equation}

Thus \eqref{eq:general-exactness} is the exactness condition selecting the
spatial directions along which the Pohozaev Rayleigh profile reconstructs
a genuine solution branch. Crucially, the orbit parameter $s$, the
Kirchhoff parameters $a$ and $b$, and the Kirchhoff exponent $\theta$ have
all disappeared from the profile equation.

For an exact direction contributing to the original problem $b\ge0$, one
necessarily has $\sigma_w>0$, since
\[
\sigma_w s^{-2}
=
a+bA(\mathcal G_s w)^\theta
>0.
\]
Define
\[
\phi(x)
=
w(\sqrt{\sigma_w}\,x).
\]
Then \eqref{eq:general-exactness} reduces to the normalized scalar-field
equation
\begin{equation}\label{eq:general-normalized-profile}
	\boxed{
		-\Delta\phi
		=
		g(\phi)
		\qquad\text{in }\mathbb R^N.
	}
\end{equation}

Conversely, let $\phi\in\mathcal D_G$ be a sufficiently regular solution of
\eqref{eq:general-normalized-profile}. Its Pohozaev identity is
\begin{equation}\label{eq:general-profile-Pohozaev}
	\frac{N-2}{2}A(\phi)
	-
	N\mathcal B(\phi)
	=
	0.
\end{equation}
Hence
\[
\sigma_\phi
=
\frac{2N\mathcal B(\phi)}{(N-2)A(\phi)}
=
1.
\]
Therefore every such normalized profile generates the extended exact spatial curve
\begin{equation}\label{eq:general-exact-branch}
	\boxed{
		\widetilde\Gamma_\phi
		=
		\left\{
		\bigl(
		b_\phi(s),\phi(s\,\cdot)
		\bigr):
		s>0
		\right\},
	}
\end{equation}
where
\begin{equation}\label{eq:general-exact-profile}
	\boxed{
		b_\phi(s)
		=
		\frac{1-as^2}{A_\phi^\theta}
		s^{\theta(N-2)-2},
		\qquad
		A_\phi:=A(\phi).
	}
\end{equation}

Thus the exact parameter law has the same form for every admissible
Berestycki--Lions nonlinearity. For the original parameter range $b\ge0$,
the relevant part of this curve is restricted below to $0<s\le a^{-1/2}$.
The nonlinearity $g$ selects the normalized profiles and their energies
$A_\phi$, whereas the dependence of the parameter on the dilation scale is
governed entirely by $N$, $\theta$, and $a$.

\begin{remark}[Universality of the spatial reduction]
	\label{rem:universality-general-profile}
	The normalized equation
	\eqref{eq:general-normalized-profile} is independent of
	$a$, $b$, and $\theta$. Consequently, for a fixed nonlinearity $g$,
	the same family of normalized scalar-field profiles generates exact
	branches for the whole class of Kirchhoff laws
	\[
	M_b(A)=a+bA^\theta,
	\qquad \theta>0.
	\]
	The nonlinearity enters each branch only through $A_\phi$, while its
	dependence on the scaling variable is always proportional to
	\[
	(1-as^2)s^{\theta(N-2)-2}.
	\]
	In particular, the dimension--homogeneity bifurcation geometry derived
	below is independent of the particular choice of $g$ within the
	admissible class.
\end{remark}

\begin{remark}[Nehari representation]
	\label{rem:general-Nehari-representation}
	If testing \eqref{eq:general-K} by $u$ is justified, every nontrivial
	solution also satisfies
	\[
	\bigl(a+bA(u)^\theta\bigr)A(u)
	=
	\int_{\mathbb R^N}g(u)u\,dx.
	\]
	Hence one may introduce
	\[
	\mathcal R_{\mathcal N,\theta}(u)
	:=
	\frac{
		\displaystyle
		\int_{\mathbb R^N}g(u)u\,dx-aA(u)
	}{
		A(u)^{\theta+1}
	}.
	\]
	Along every exact spatial branch generated by
	\eqref{eq:general-normalized-profile},
	\[
	\mathcal R_{\mathcal N,\theta}(\mathcal G_s\phi)
	=
	\mathcal R_{\mathcal P,\theta}(\mathcal G_s\phi)
	=
	b_\phi(s).
	\]
	Thus the Nehari representation yields the same parameter law on the
	exact branch, although it is not needed for the spatial reduction.
\end{remark}

\subsection{Dimension--homogeneity bifurcation trichotomy}
\label{subsec:general-exact-geometry}

Let $\phi\in\mathcal D_G$ be a sufficiently regular solution of
\eqref{eq:general-normalized-profile}, and set
\[
A_\phi:=A(\phi).
\]
By \eqref{eq:general-exact-profile}, the corresponding exact spatial branch
is governed by
\begin{equation}\label{eq:general-exact-profile-geometry}
	\boxed{
		b_\phi(s)
		=
		\frac{1-as^2}{A_\phi^\theta}
		s^{\theta(N-2)-2},
		\qquad s>0.
	}
\end{equation}

\begin{remark}[Generalized Azzollini rescaling relation]
	\label{rem:generalized-Azzollini-relation}
	For $u=\phi(s\,\cdot)$,
	\[
	A(u)^\theta
	=
	s^{-\theta(N-2)}A_\phi^\theta.
	\]
	Hence the exact rescaling relation is equivalent to
	\begin{equation}\label{eq:generalized-Azzollini-relation}
		\boxed{
			as^2
			+
			bA_\phi^\theta
			s^{\,2-\theta(N-2)}
			=
			1.
		}
	\end{equation}
	Solving for $b$ recovers
	\eqref{eq:general-exact-profile-geometry}. For $\theta=1$, this
	reduces to
	\[
	as^2+bA_\phi s^{4-N}=1,
	\]
	the usual Azzollini-type rescaling relation
	\cite{Azzollini2012,Azzollini2015} for the standard Kirchhoff law.
\end{remark}

For the original problem $b\ge0$, the relevant part of the exact curve is
\[
\Gamma_\phi
=
\left\{
\bigl(
b_\phi(s),\phi(s\,\cdot)
\bigr):
0<s\le s_{\rm z}
\right\},
\qquad
s_{\rm z}:=a^{-1/2},
\]
since
\[
b_\phi(s)\ge0
\quad\Longleftrightarrow\quad
0<s\le s_{\rm z}.
\]

The geometry of every exact spatial branch is governed by the
dimension--homogeneity index
\begin{equation}\label{eq:dimension-homogeneity-index}
	\boxed{
		k:=\theta(N-2)-2.
	}
\end{equation}
Indeed,
\[
b_\phi(s)
=
\frac{1-as^2}{A_\phi^\theta}s^k,
\]
and
\begin{equation}\label{eq:general-profile-derivative}
	b_\phi'(s)
	=
	\frac{s^{k-1}}{A_\phi^\theta}
	\left[
	k-a(k+2)s^2
	\right].
\end{equation}
The factor $A_\phi^{-\theta}$ changes only the vertical scale of the
profile. Its qualitative geometry is therefore determined solely by
$N$ and $\theta$. This dimension--homogeneity balance leads to three
distinct regimes, illustrated in Figure~\ref{fig:trichotomy-geometry}.

\begin{figure}[t]
	\centering
	\includegraphics[width=.82\textwidth]{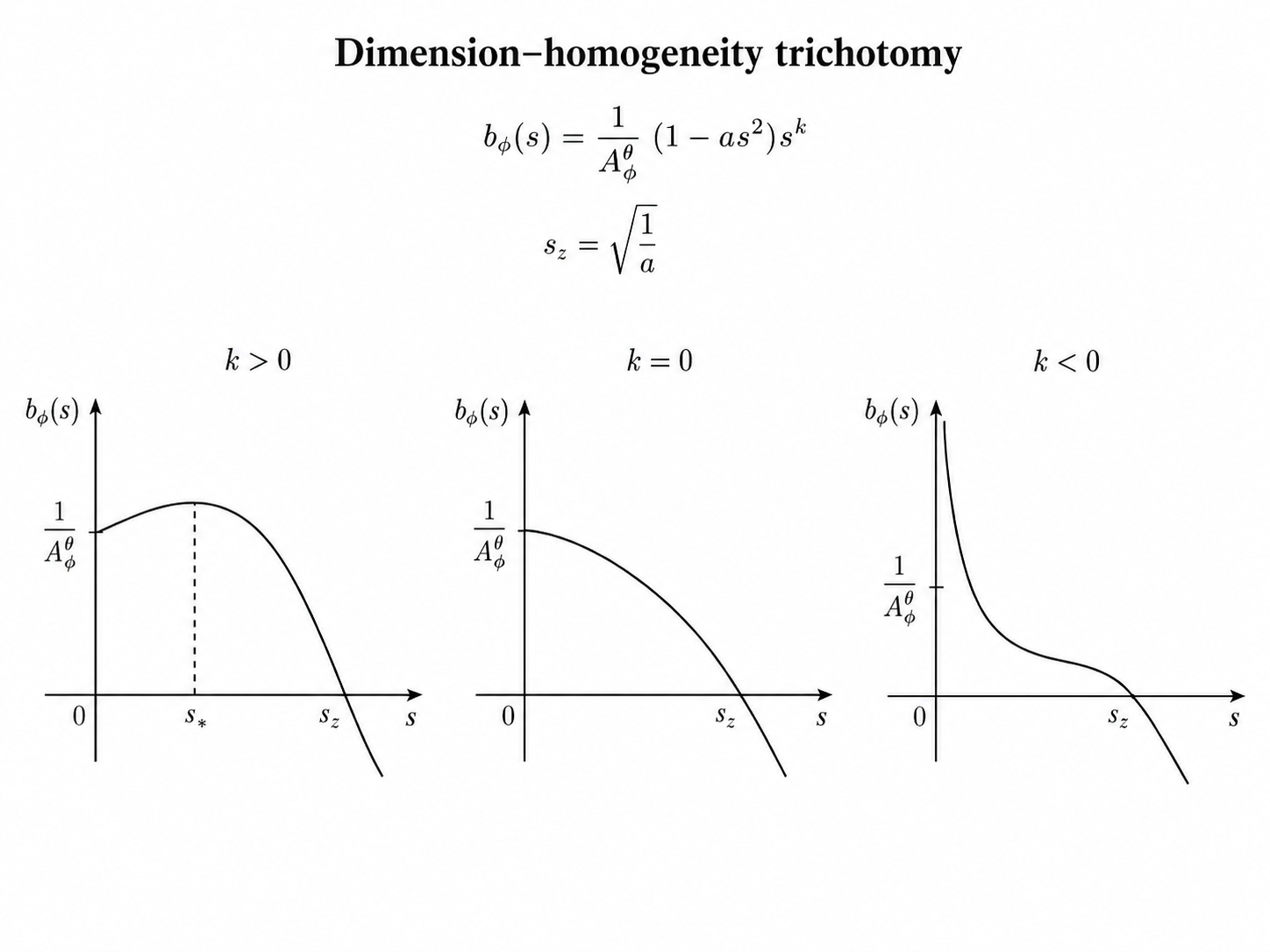}
	\caption{Dimension--homogeneity trichotomy for the exact spatial
		Rayleigh profile
		$b_\phi(s)=A_\phi^{-\theta}(1-as^2)s^k$,
		where $k=\theta(N-2)-2$.
		For $k>0$ the profile has a unique interior maximum $s_*$;
		for $k=0$ it decreases monotonically from a finite nonattained
		upper value; and for $k<0$ it is singular as $s\downarrow0$ and
		decreases monotonically. In all three cases
		$b_\phi(s_{\rm z})=0$, with $s_{\rm z}=a^{-1/2}$.}
	\label{fig:trichotomy-geometry}
\end{figure}

\medskip
\noindent
\textbf{Supercritical scaling regime: $\theta(N-2)>2$.}
Here $k>0$ and
\[
b_\phi(s)\to0
\quad\text{as }s\downarrow0,
\qquad
b_\phi(s_{\rm z})=0.
\]
The equation $b_\phi'(s)=0$ has the unique solution
\begin{equation}\label{eq:general-s-star}
	\boxed{
		s_*
		=
		\left(
		\frac{\theta(N-2)-2}
		{a\theta(N-2)}
		\right)^{1/2},
	}
\end{equation}
with
\[
0<s_*<s_{\rm z}.
\]
Moreover,
\begin{equation}\label{eq:general-second-derivative}
	b_\phi''(s_*)
	=
	-\frac{2a\theta(N-2)}{A_\phi^\theta}
	s_*^{\theta(N-2)-2}
	<0.
\end{equation}
Hence $s_*$ is the unique nondegenerate global maximum of the exact
Rayleigh profile, with
\begin{equation}\label{eq:general-critical-value}
	\boxed{
		b_\phi^*
		=
		\frac{2}
		{\theta(N-2)A_\phi^\theta}
		\left(
		\frac{\theta(N-2)-2}
		{a\theta(N-2)}
		\right)^{\frac{\theta(N-2)-2}{2}}.
	}
\end{equation}

The dilation orbit
\[
s\longmapsto\phi(s\,\cdot)
\]
is regular for $\phi\neq0$: otherwise $x\cdot\nabla\phi=0$, which is
incompatible with a nonzero $H^1(\mathbb R^N)$ profile. Therefore, by
Theorem~\ref{thm:Rayleigh-bifurcation},
\[
\bigl(
b_\phi^*,\phi(s_*\cdot)
\bigr)
\]
is a nondegenerate turning bifurcation point of the exact spatial branch.

\medskip
\noindent
\textbf{Critical scaling regime: $\theta(N-2)=2$.}
Here
\[
b_\phi(s)
=
\frac{1-as^2}{A_\phi^\theta}
\]
is strictly decreasing on $(0,s_{\rm z}]$, with
\[
\lim_{s\downarrow0}b_\phi(s)
=
A_\phi^{-\theta},
\qquad
b_\phi(s_{\rm z})=0.
\]
Thus $A_\phi^{-\theta}$ is a nonattained upper threshold, and there is no
interior turning point.

\medskip
\noindent
\textbf{Subcritical scaling regime: $\theta(N-2)<2$.}
Here $k<0$ and, since
\[
k+2=\theta(N-2)>0,
\]
we have
\[
k-a(k+2)s^2<0,
\qquad s>0.
\]
Hence $b_\phi$ is strictly decreasing, with
\[
b_\phi(s)\to+\infty
\quad\text{as }s\downarrow0,
\qquad
b_\phi(s_{\rm z})=0.
\]
Again there is no interior turning point.

We have therefore obtained the following universal spatial bifurcation
trichotomy.

\begin{theorem}[Dimension--homogeneity trichotomy]
	\label{thm:dimension-homogeneity-trichotomy}
	Let $\phi\in\mathcal D_G$ be a sufficiently regular solution of
	\eqref{eq:general-normalized-profile}. Then:
	
	\begin{enumerate}[label=\textup{(\roman*)}]
		\item if $\theta(N-2)>2$, the exact spatial branch
		$\Gamma_\phi$ has a unique nondegenerate turning bifurcation point
		\[
		\bigl(
		b_\phi^*,\phi(s_*\cdot)
		\bigr).
		\]
		The branch contains exactly two points for
		$0<b<b_\phi^*$, one point for $b=0$, one turning solution for
		$b=b_\phi^*$, and none for $b>b_\phi^*$;
		
		\item if $\theta(N-2)=2$, the branch contains exactly one point
		for
		\[
		0\le b<A_\phi^{-\theta},
		\]
		and none for $b\ge A_\phi^{-\theta}$;
		
		\item if $\theta(N-2)<2$, the branch contains exactly one point
		for every $b\ge0$.
	\end{enumerate}
	
	In particular,
	\[
	\boxed{
		\theta(N-2)>2
	}
	\]
	is the sharp dimension--homogeneity condition for the occurrence of
	an interior turning bifurcation on an exact spatial branch. This
	criterion is independent of the particular Berestycki--Lions
	nonlinearity $g$.
\end{theorem}

\subsection{Ground states and the global threshold}
\label{subsec:profile-Rayleigh-maximal-branch}
Let
\[
\mathcal P
:=
\left\{
\phi\in\mathcal D_G:
-\Delta\phi=g(\phi)
\ \text{in }\mathbb R^N
\right\},
\]
and
\[
\mathcal P_{\rm rad}
:=
\mathcal P\cap H^1_{\rm rad}(\mathbb R^N).
\]
Define the scalar-field action
\[
J(\phi)
:=
\frac12A(\phi)-\mathcal B(\phi).
\]
Under the Berestycki--Lions assumptions, the normalized problem admits a
positive radial ground state; see \cite{BerestyckiLions1983I}. Let
$U\in\mathcal P_{\rm rad}$ be any such ground state and set
\[
c_g
:=
J(U)
=
\inf_{\phi\in\mathcal P}J(\phi),
\qquad
A_g:=A(U).
\]

For every $\phi\in\mathcal P$, the Pohozaev identity
\[
\frac{N-2}{2}A(\phi)-N\mathcal B(\phi)=0
\]
gives
\begin{equation}\label{eq:general-action-energy}
	\boxed{
		J(\phi)=\frac1N A(\phi).
	}
\end{equation}
Consequently,
\begin{equation}\label{eq:ground-state-minimal-energy}
	\boxed{
		A_g
		=
		\min_{\phi\in\mathcal P}A(\phi)
		=
		Nc_g.
	}
\end{equation}
In particular, all ground states have the same gradient energy $A_g$.

By the exact spatial reduction, every $\phi\in\mathcal P$ generates the
branch
\[
\Gamma_\phi
=
\left\{
\bigl(
b_\phi(s),\phi(s\,\cdot)
\bigr):
0<s\le s_{\rm z}
\right\},
\]
where
\[
b_\phi(s)
=
\frac{1-as^2}{A(\phi)^\theta}
s^{\theta(N-2)-2}.
\]
Thus every ground state generates the same scalar parameter profile
\begin{equation}\label{eq:ground-state-profile}
	\boxed{
		b_g(s)
		=
		\frac{1-as^2}{A_g^\theta}
		s^{\theta(N-2)-2}.
	}
\end{equation}
The corresponding solution branches themselves need not coincide when the
normalized ground state is nonunique.

Conversely, the exact spatial class exhausts all nontrivial finite-energy
solutions of the Kirchhoff equation. Indeed, let $u\neq0$ solve
\eqref{eq:general-K} and set
\[
\kappa:=a+bA(u)^\theta.
\]
Then
\[
\phi(x):=u(\sqrt{\kappa}\,x)
\]
satisfies
\[
-\Delta\phi=g(\phi).
\]
Since $u\in\mathcal D_G$ and $\mathcal D_G$ is invariant under spatial
dilations, we also have $\phi\in\mathcal D_G$, and hence
$\phi\in\mathcal P$. Moreover,
\[
u(x)=\phi(sx),
\qquad
s=\kappa^{-1/2}.
\]
Since $b\ge0$,
\[
\kappa\ge a,
\qquad
0<s\le s_{\rm z}=a^{-1/2}.
\]
Thus every nontrivial finite-energy solution belongs to one of the exact
branches $\Gamma_\phi$.

Assume first that
\[
\theta(N-2)>2.
\]
By \eqref{eq:general-critical-value},
\[
b_\phi^*
=
\frac{C_{a,N,\theta}}{A(\phi)^\theta},
\]
where
\begin{equation}\label{eq:general-C-antheta}
	C_{a,N,\theta}
	:=
	\frac{2}{\theta(N-2)}
	\left(
	\frac{\theta(N-2)-2}
	{a\theta(N-2)}
	\right)^{\frac{\theta(N-2)-2}{2}}.
\end{equation}
Since the ground states minimize $A(\phi)$ on $\mathcal P$, they maximize
the turning value. Hence
\begin{equation}\label{eq:maximal-ground-state-branch}
	\boxed{
		b_g^*
		:=
		\sup_{\phi\in\mathcal P}b_\phi^*
		=
		\frac{C_{a,N,\theta}}{A_g^\theta}.
	}
\end{equation}
Equivalently,
\begin{equation}\label{eq:bg-action-level}
	\boxed{
		b_g^*
		=
		\frac{C_{a,N,\theta}}
		{(Nc_g)^\theta}.
	}
\end{equation}

Thus every ground-state profile generates the maximal turning parameter.
Moreover, if $u\neq0$ is any solution, then
\[
u=\phi(s\,\cdot)
\]
for some $\phi\in\mathcal P$, and therefore
\[
b=b_\phi(s)
\le b_\phi^*
\le b_g^*.
\]
Consequently, $b_g^*$ is the global upper existence threshold:
\begin{equation}\label{eq:global-upper-threshold-general}
	\boxed{
		b>b_g^*
		\quad\Longrightarrow\quad
		\text{\eqref{eq:general-K} has no nontrivial solution}.
	}
\end{equation}

The same ground-state selection determines the global threshold in the
critical regime. If
\[
\theta(N-2)=2,
\]
then the branch generated by $\phi\in\mathcal P$ exists precisely for
\[
0\le b<A(\phi)^{-\theta}.
\]
Since $A(\phi)\ge A_g$, the value
\begin{equation}\label{eq:critical-global-threshold}
	\boxed{
		b_g^{\,0}
		:=
		A_g^{-\theta}
		=
		(Nc_g)^{-\theta}
	}
\end{equation}
is the sharp nonattained global upper existence threshold. Every
ground-state branch realizes
\[
0\le b<b_g^{\,0},
\]
whereas no nontrivial solution exists for
\[
b\ge b_g^{\,0}.
\]

If
\[
\theta(N-2)<2,
\]
there is no finite upper parameter threshold: every ground-state branch
contains exactly one solution for each $b\ge0$.

We summarize these observations as follows.

\begin{corollary}[Ground-state maximality]
	\label{cor:ground-state-maximality}
	Let $U$ be any ground state of
	\eqref{eq:general-normalized-profile}, and let
	$A_g=A(U)=Nc_g$.
	
	\begin{enumerate}[label=\textup{(\roman*)}]
		\item If $\theta(N-2)>2$, every ground-state branch has the same
		nondegenerate turning value
		\[
		b_g^*
		=
		\frac{C_{a,N,\theta}}{A_g^\theta}
		=
		\sup_{\phi\in\mathcal P}b_\phi^*,
		\]
		and $b_g^*$ is the global upper existence threshold.
		
		\item If $\theta(N-2)=2$, the value
		\[
		b_g^{\,0}=A_g^{-\theta}
		\]
		is the sharp nonattained global upper existence threshold.
		
		\item If $\theta(N-2)<2$, the Kirchhoff problem has no finite
		upper threshold, and every ground-state branch contains exactly
		one solution for every $b\ge0$.
	\end{enumerate}
\end{corollary}

For the standard Kirchhoff law $\theta=1$, the governing exponent becomes
\[
k=N-4.
\]
Hence the classical distinction between $N=3$, $N=4$, and $N\ge5$ is
precisely the specialization of the dimension--homogeneity balance.

\begin{corollary}[Standard Kirchhoff law]
	\label{cor:standard-Kirchhoff-trichotomy}
	For $\theta=1$:
	\begin{enumerate}[label=\textup{(\roman*)}]
		\item if $N\ge5$, every exact spatial branch has a unique
		nondegenerate turning bifurcation point, with
		\[
		s_*
		=
		\left(
		\frac{N-4}{a(N-2)}
		\right)^{1/2},
		\qquad
		b_\phi^*
		=
		\frac{2}{A_\phi(N-2)}
		\left(
		\frac{N-4}{a(N-2)}
		\right)^{\frac{N-4}{2}};
		\]
		
		\item if $N=4$, every branch is strictly decreasing from the
		nonattained value $A_\phi^{-1}$ to $0$;
		
		\item if $N=3$, every branch is strictly decreasing from
		$+\infty$ to $0$.
	\end{enumerate}
\end{corollary}

\begin{remark}[Positive radial solutions and the power case]
	\label{rem:positive-radial-power-case}
	For a general Berestycki--Lions nonlinearity, uniqueness of a positive
	radial normalized profile is not assumed. Hence the positive radial
	solution set of the Kirchhoff problem is, in general, the union of the
	exact branches generated by all positive profiles in
	$\mathcal P_{\rm rad}$.
	
	For the power nonlinearity
	\[
	g(t)=|t|^{p-2}t-\mu t,
	\qquad
	2<p<2^*,
	\]
	Kwong's theorem \cite{Kwong1989} gives uniqueness of the positive
	radial normalized profile $U_\mu$. In this case the positive radial
	solution set reduces to a single exact spatial branch.
\end{remark}

\subsection{Variational hierarchy of turning thresholds}
\label{subsec:countable-general-branches}

Continue to assume
\[
\theta(N-2)>2.
\]
Under the Berestycki--Lions assumptions, the normalized scalar-field
equation possesses infinitely many distinct radial solution pairs; see
\cite{BerestyckiLions1983II}. Let $\phi_1=U$ be a ground state and choose
the remaining profiles from a radial variational sequence so that
\[
\{\pm\phi_j\}_{j\ge1}
\subset\mathcal P_{\rm rad},
\qquad
J(\phi_j)\longrightarrow+\infty.
\]
Since
\[
J(\phi)=\frac1N A(\phi)
\qquad\text{on }\mathcal P,
\]
after passing to a subsequence and relabeling we may assume
\[
A_j:=A(\phi_j),
\qquad
A_1=A_g\le A_2\le\cdots,
\qquad
A_j\longrightarrow+\infty.
\]

Each profile generates the exact branch
\[
\Gamma_j
=
\left\{
\bigl(
b_j(s),\phi_j(s\,\cdot)
\bigr):
0<s\le s_{\rm z}
\right\},
\]
where
\begin{equation}\label{eq:general-countable-branches}
	b_j(s)
	=
	\frac{1-as^2}{A_j^\theta}
	s^{\theta(N-2)-2}.
\end{equation}
All branches have the same critical spatial scale $s_*$, while their
turning values are
\begin{equation}\label{eq:general-countable-bstar}
	\boxed{
		b_j^*
		=
		\frac{2}
		{\theta(N-2)A_j^\theta}
		\left(
		\frac{\theta(N-2)-2}
		{a\theta(N-2)}
		\right)^{\frac{\theta(N-2)-2}{2}}.
	}
\end{equation}
Hence
\begin{equation}\label{eq:general-countable-order}
	\boxed{
		b_1^*
		=
		b_g^*
		\ge
		b_2^*
		\ge
		\cdots
		\longrightarrow0.
	}
\end{equation}

The branches are mutually distinct. Indeed, suppose that
\[
\phi_i(s\,\cdot)
=
\pm\phi_j(r\,\cdot)
=:u.
\]
Since $g$ is odd and both $\phi_i$ and $\phi_j$ solve the normalized
equation, $u$ satisfies
\[
-s^{-2}\Delta u=g(u),
\qquad
-r^{-2}\Delta u=g(u).
\]
Subtracting gives
\[
(s^{-2}-r^{-2})\Delta u=0.
\]
If $s\neq r$, then $\Delta u=0$. Since
$u\in H^1(\mathbb R^N)$, this implies $u=0$, a contradiction.
Hence $s=r$, and consequently
\[
\phi_i=\pm\phi_j.
\]

Let
\[
\mathcal S_{\rm ex}
:=
\left\{
\mathcal G_s\phi:
\phi\in\mathcal P,\ 
0<s\le s_{\rm z}
\right\}.
\]
Define the upper Rayleigh threshold of the exact spatial class by
\begin{equation}\label{eq:general-exact-threshold-definition}
	b_{\mathcal G,\mathcal S_{\rm ex}}^{**}
	:=
	\sup_{\phi\in\mathcal P}
	\sup_{0<s\le s_{\rm z}}
	\mathcal R_{\mathcal P,\theta}(\mathcal G_s\phi).
\end{equation}
Since
\[
\mathcal R_{\mathcal P,\theta}(\mathcal G_s\phi)
=
b_\phi(s),
\]
and the ground states minimize $A(\phi)$ on $\mathcal P$, we obtain
\begin{equation}\label{eq:general-max-threshold}
	\boxed{
		b_{\mathcal G,\mathcal S_{\rm ex}}^{**}
		=
		b_1^*
		=
		b_g^*.
	}
\end{equation}
By the exhaustivity established in
Subsection~\ref{subsec:profile-Rayleigh-maximal-branch}, this is also the
global upper existence threshold for \eqref{eq:general-K}.

Thus the variational hierarchy of the normalized scalar-field equation is
transformed by the exact Rayleigh reduction into an inverse hierarchy of
Kirchhoff turning thresholds.

\begin{theorem}[Spatial bifurcation hierarchy]
	\label{thm:general-spatial-hierarchy}
	Assume $\theta(N-2)>2$. Then \eqref{eq:general-K} possesses infinitely
	many distinct exact radial branches $\Gamma_j$, each with a unique
	nondegenerate turning bifurcation point
	\[
	\bigl(
	b_j^*,\phi_j(s_*\cdot)
	\bigr).
	\]
	The corresponding turning values satisfy
	\[
	\boxed{
		b_{\mathcal G,\mathcal S_{\rm ex}}^{**}
		=
		b_1^*
		=
		b_g^*
		\ge
		b_2^*
		\ge
		\cdots
		\longrightarrow0.
	}
	\]
	
	More precisely:
	\begin{enumerate}[label=\textup{(\roman*)}]
		\item for $b=0$, problem \eqref{eq:general-K} possesses infinitely
		many distinct antipodal pairs of nontrivial radial solutions;
		
		\item for every $m\in\mathbb N$ and
		\[
		0<b<b_m^*,
		\]
		it possesses at least $2m$ distinct antipodal pairs of nontrivial
		radial solutions;
		
		\item for
		\[
		b>b_g^*
		=
		b_{\mathcal G,\mathcal S_{\rm ex}}^{**},
		\]
		problem \eqref{eq:general-K} has no nontrivial solutions.
	\end{enumerate}
\end{theorem}

\begin{proof}
	The distinctness of the branches was proved above, and
	\eqref{eq:general-countable-bstar}--\eqref{eq:general-countable-order}
	give
	\[
	b_1^*
	=
	b_g^*
	\ge
	b_2^*
	\ge
	\cdots
	\longrightarrow0.
	\]
	
	For every $j$, the profile
	\[
	b_j(s)
	=
	\frac{1-as^2}{A_j^\theta}
	s^{\theta(N-2)-2}
	\]
	has the same unique nondegenerate maximum at $s=s_*$. The dilation
	orbit
	\[
	s\longmapsto\phi_j(s\,\cdot)
	\]
	is regular for $\phi_j\neq0$; otherwise
	$x\cdot\nabla\phi_j=0$, which is incompatible with a nonzero
	$H^1(\mathbb R^N)$ profile. Hence
	Theorem~\ref{thm:Rayleigh-bifurcation} yields the nondegenerate turning
	bifurcation point
	\[
	\bigl(
	b_j^*,\phi_j(s_*\cdot)
	\bigr).
	\]
	
	For $b=0$, each profile pair gives the solution pair
	\[
	\{\pm\phi_j(s_{\rm z}\,\cdot)\},
	\]
	and the distinctness of the branches proves \textup{(i)}.
	
	Now let $m\in\mathbb N$ and
	\[
	0<b<b_m^*.
	\]
	For every $j\le m$,
	\[
	b<b_m^*\le b_j^*.
	\]
	Hence
	\[
	b_j(s)=b
	\]
	has two distinct roots
	\[
	s_j^-(b)<s_*<s_j^+(b).
	\]
	Thus each of the first $m$ profile pairs produces two distinct
	antipodal solution pairs
	\[
	\{\pm\phi_j(s_j^-(b)\,\cdot)\},
	\qquad
	\{\pm\phi_j(s_j^+(b)\,\cdot)\}.
	\]
	Therefore \eqref{eq:general-K} has at least $2m$ distinct antipodal
	pairs, proving \textup{(ii)}.
	
	Finally, \eqref{eq:general-max-threshold} and the exhaustivity of the
	exact spatial class imply
	\[
	b\le b_g^*
	\]
	for every nontrivial solution. This proves \textup{(iii)}.
\end{proof}

\subsection{Radial simple-fold classification}
\label{subsec:general-radial-fold}

The exact Rayleigh reduction detects turning bifurcation points without
requiring any spectral analysis of the linearization. Under additional
regularity and nondegeneracy assumptions on a normalized radial profile,
the corresponding turning point can also be classified as a radial simple
fold.

\begin{corollary}[Radial simple fold]
	\label{cor:general-radial-simple-fold}
	Assume
	\[
	\theta(N-2)>2,
	\]
	and let $\phi\in\mathcal P_{\rm rad}$. Set
	\[
	u_*:=\phi(s_*\cdot).
	\]
	Suppose, in addition, that $g\in C^1(\mathbb R)$, that the Nemytskii
	operator induced by $g$ yields a $C^1$ realization
	\[
	F_\theta:
	\mathbb R\times H^1_{\rm rad}(\mathbb R^N)
	\longrightarrow
	H^{-1}_{\rm rad}(\mathbb R^N)
	\]
	in a neighborhood of $(b_\phi^*,u_*)$, and that
	\[
	s\longmapsto\phi(s\,\cdot)
	\]
	is a $C^1$ regular curve in $H^1_{\rm rad}(\mathbb R^N)$ near $s_*$.
	Assume also that the scalar-field linearization
	\[
	\mathcal L_\phi
	:=
	-\Delta-g'(\phi)
	:
	H^1_{\rm rad}(\mathbb R^N)
	\longrightarrow
	H^{-1}_{\rm rad}(\mathbb R^N)
	\]
	is Fredholm of index zero and radially nondegenerate:
	\[
	\ker\mathcal L_\phi=\{0\}.
	\]
	
	Then $(b_\phi^*,u_*)$ is a radial simple fold of
	\eqref{eq:general-K}. More precisely, in a neighborhood of
	$(b_\phi^*,u_*)$, the radial solution set is a one-dimensional
	$C^1$ manifold locally parametrized by the exact spatial branch
	\[
	s\longmapsto
	\bigl(
	b_\phi(s),\phi(s\,\cdot)
	\bigr),
	\]
	and its projection onto the parameter axis has a nondegenerate
	quadratic maximum at $s=s_*$. Moreover,
	\[
	\boxed{
		\ker
		\left(
		F_{\theta,u}(b_\phi^*,u_*)
		\big|_{H^1_{\rm rad}(\mathbb R^N)}
		\right)
		=
		\operatorname{span}
		\left\{
		\left.
		\frac{d}{ds}\phi(s\,\cdot)
		\right|_{s=s_*}
		\right\}.
	}
	\]
\end{corollary}

\begin{proof}
	Set
	\[
	u_s:=\phi(s\,\cdot),
	\qquad
	\dot u_*:=
	\left.
	\frac{d}{ds}u_s
	\right|_{s=s_*}.
	\]
	By the regularity assumption on the dilation orbit,
	\[
	\dot u_*\neq0.
	\]
	Differentiating
	\[
	F_\theta\bigl(b_\phi(s),u_s\bigr)=0
	\]
	at $s=s_*$ and using $b_\phi'(s_*)=0$ gives
	\begin{equation}\label{eq:general-fold-kernel-tangent}
		F_{\theta,u}(b_\phi^*,u_*)\dot u_*=0.
	\end{equation}
	
	For radial $v$,
	\[
	\begin{aligned}
		F_{\theta,u}(b,u)v
		={}&
		-\bigl(a+bA(u)^\theta\bigr)\Delta v
		-g'(u)v
		\\
		&\quad
		-2b\theta A(u)^{\theta-1}
		\left(
		\int_{\mathbb R^N}
		\nabla u\cdot\nabla v\,dx
		\right)\Delta u.
	\end{aligned}
	\]
	Let
	\[
	L_*:=F_{\theta,u}(b_\phi^*,u_*):
	H^1_{\rm rad}(\mathbb R^N)
	\longrightarrow
	H^{-1}_{\rm rad}(\mathbb R^N),
	\]
	and write
	\[
	L_*=B_*+K_*,
	\]
	where
	\[
	B_*v
	=
	-\bigl(a+b_\phi^*A(u_*)^\theta\bigr)\Delta v
	-g'(u_*)v
	\]
	and $K_*$ is the rank-one nonlocal term.
	
	By the exact rescaling relation,
	\[
	a+b_\phi^*A(u_*)^\theta=s_*^{-2}.
	\]
	Under the dilation $y=s_*x$, the operator $B_*$ is conjugate to
	\[
	\mathcal L_\phi=-\Delta-g'(\phi).
	\]
	The Fredholm index-zero property together with
	\[
	\ker\mathcal L_\phi=\{0\}
	\]
	implies that $\mathcal L_\phi$, and hence $B_*$, is invertible.
	Since $K_*$ has rank one, $L_*$ is Fredholm of index zero and
	\[
	\dim\ker L_*\le1.
	\]
	Together with \eqref{eq:general-fold-kernel-tangent} and
	$\dot u_*\neq0$, this yields
	\[
	\ker L_*
	=
	\operatorname{span}\{\dot u_*\}.
	\]
	
	It remains to verify parameter transversality. Since
	\[
	F_{\theta,b}(b,u)
	=
	-A(u)^\theta\Delta u,
	\]
	we have
	\[
	\begin{aligned}
		\left\langle
		F_{\theta,b}(b_\phi^*,u_*),\dot u_*
		\right\rangle
		&=
		A(u_*)^\theta
		\int_{\mathbb R^N}
		\nabla u_*\cdot\nabla\dot u_*\,dx
		\\
		&=
		\frac{A(u_*)^\theta}{2}
		\left.
		\frac{d}{ds}A(u_s)
		\right|_{s=s_*}.
	\end{aligned}
	\]
	Because
	\[
	A(u_s)=s^{2-N}A_\phi,
	\]
	it follows that
	\[
	\left.
	\frac{d}{ds}A(u_s)
	\right|_{s=s_*}
	=
	(2-N)A_\phi s_*^{1-N}
	\neq0.
	\]
	Thus
	\begin{equation}\label{eq:general-fold-transversality-pairing}
		\left\langle
		F_{\theta,b}(b_\phi^*,u_*),\dot u_*
		\right\rangle
		\neq0.
	\end{equation}
	
	The bilinear form associated with $L_*$ is symmetric. Since $L_*$ is
	Fredholm of index zero and
	\[
	\ker L_*=\operatorname{span}\{\dot u_*\},
	\]
	we have
	\[
	\operatorname{Ran}L_*
	=
	\left\{
	f\in H^{-1}_{\rm rad}(\mathbb R^N):
	\langle f,\dot u_*\rangle=0
	\right\}.
	\]
	Hence
	\eqref{eq:general-fold-transversality-pairing} implies
	\[
	F_{\theta,b}(b_\phi^*,u_*)
	\notin
	\operatorname{Ran}L_*.
	\]
	
	Consequently, the full derivative
	\[
	D F_\theta(b_\phi^*,u_*)[\beta,v]
	=
	\beta F_{\theta,b}(b_\phi^*,u_*)
	+
	L_*v
	\]
	is surjective, with
	\[
	\ker D F_\theta(b_\phi^*,u_*)
	=
	\operatorname{span}\{(0,\dot u_*)\}.
	\]
	By the regular level-set theorem, the radial solution set near
	$(b_\phi^*,u_*)$ is therefore a one-dimensional $C^1$ manifold.
	
	The exact curve
	\[
	\gamma(s)
	=
	\bigl(b_\phi(s),u_s\bigr)
	\]
	lies in this manifold and satisfies
	\[
	\gamma'(s_*)
	=
	(0,\dot u_*)
	\neq0.
	\]
	It therefore gives a local parametrization of the radial solution
	manifold. Finally,
	\[
	b_\phi'(s_*)=0,
	\qquad
	b_\phi''(s_*)<0,
	\]
	so the projection onto the parameter axis has a nondegenerate
	quadratic maximum at $s=s_*$. Hence
	$(b_\phi^*,u_*)$ is a radial simple fold.
\end{proof}

\begin{remark}[Role of the additional assumptions]
	The $C^1$ regularity and radial nondegeneracy assumptions above are
	not part of the basic Berestycki--Lions hypotheses and are needed
	only for the local simple-fold classification. They play no role in
	the exact Rayleigh detection of turning bifurcations, in the
	dimension--homogeneity trichotomy, or in the variational hierarchy.
	Whenever the operator realization is of class $C^2$, the preceding
	geometric fold agrees with the classical Lyapunov--Schmidt
	simple-fold criterion discussed in
	Subsection~\ref{subsec:Rayleigh-classical-fold}.
\end{remark}

\begin{remark}[Translation modes]
	In the full space $H^1(\mathbb R^N)$, translation invariance produces
	the additional kernel directions
	\[
	\partial_{x_1}u_*,
	\ldots,
	\partial_{x_N}u_*.
	\]
	Hence the one-dimensional-kernel formulation is naturally made in
	$H^1_{\rm rad}(\mathbb R^N)$, or after factoring out translations.
\end{remark}

\section{Discussion}
\label{sec:discussion}

\subsection{Beyond amplitude and spatial scaling}

The examples above show that amplitude scaling and spatial dilation are
particular realizations of a more general mechanism. Its essential
ingredients are a one-parameter transformation orbit
\[
u(s)=\mathcal G_s\phi
\]
and a generalized Rayleigh functional for which
\begin{equation}\label{eq:discussion-Rayleigh}
	\lambda_\phi(s)
	=
	\mathcal R(\mathcal G_s\phi)
\end{equation}
becomes the actual parameter law of an exact solution branch:
\begin{equation}\label{eq:discussion-exactness}
	F\bigl(
	\mathcal R(\mathcal G_s\phi),
	\mathcal G_s\phi
	\bigr)=0.
\end{equation}

The decisive point is exactness. A critical point of an arbitrary fiber map
or generalized Rayleigh functional may reflect only a degeneracy of an
auxiliary constraint. Once \eqref{eq:discussion-exactness} holds, however,
the scalar Rayleigh geometry becomes part of the geometry of the genuine
solution set. Nondegenerate critical points of the exact Rayleigh profile
then detect turning bifurcations directly through
\[
\frac{d}{ds}\mathcal R(\mathcal G_s\phi)=0.
\]

This changes the role of the Rayleigh functional. Instead of merely
assigning a distinguished parameter level to a state, it may reconstruct
the parameter evolution along a nonlinear solution branch. Consequently,
multiplicity, turning thresholds, and branch hierarchies can become visible
through the geometry of a one-dimensional profile.

The weighted bounded-domain problem illustrates an important robustness of
this mechanism. The introduction of a positive spatial coefficient $k(x)$
changes the normalized profile equation and hence the profile energies and
turning levels, but leaves the exact amplitude geometry unchanged. In
particular, the critical amplitude
\[
s_*
=
\left(
\frac{2a}{4-q}
\right)^{1/(q-2)}
\]
is independent of both the profile and the spatial weight. Thus exact
Rayleigh reduction need not rely on complete spatial homogeneity of the
underlying equation.

This viewpoint suggests a broader strategy for nonlinear bifurcation
problems: search simultaneously for a scalar parameter representation and
a transformation family for which the corresponding exactness condition
reduces to a tractable normalized profile equation. The transformation need
not be a pure amplitude or spatial scaling. Natural candidates include
combined transformations
\[
u_s(x)=s^\alpha u(sx),
\]
anisotropic dilations, and transformations adapted to conservation laws,
homogeneity, or other geometric structures of the equation. For example,
$\alpha=N/2$ preserves the $L^2$ norm and is natural in normalized
nonlinear Schr\"odinger problems and related models; see
\cite{BartschSoave2017,BellazziniSiciliano2011,Ye2015b}.

The method also allows several scalar constraints to coexist. Their
generalized Rayleigh functionals may differ on the ambient state space but
recover the same external parameter along a common exact orbit. The
coincidence of the Nehari and Pohozaev Rayleigh functionals on the exact
spatial Kirchhoff branches provides a concrete example.

It is useful to distinguish this mechanism from an ordinary symmetry
transformation. A symmetry maps solutions of a fixed equation to solutions
of the same equation, while an equivalence transformation prescribes the
change of parameters together with the transformation. In an exact Rayleigh
reduction, by contrast, the parameter evolution is reconstructed from the
transformed state itself:
\[
\lambda_\phi(s)=\mathcal R(\mathcal G_s\phi).
\]
Its geometry is therefore not imposed in advance. In particular,
nonmonotonicity of the resulting parameter law may reveal turning
bifurcations and parameter thresholds that are hidden in the original
operator equation.

The broader problem suggested by this work is therefore to identify
nonlinear equations, scalar constraints, and transformation families for
which a generalized Rayleigh representation admits such an exact
realization.

\subsection{Linearization, spectral transitions, and stability}

Exact Rayleigh reduction also creates a direct link between branch geometry
and the spectrum of the linearized equation. At a critical point of an
exact branch,
\[
F(\lambda_\phi(s),u(s))=0,
\qquad
\lambda_\phi'(s_*)=0,
\]
differentiation gives
\[
F_u(\lambda_*,u_*)u'(s_*)=0.
\]
Thus the tangent to the explicitly constructed orbit provides a kernel
direction of the state linearization. The loss of invertibility is not used
to locate the critical parameter; rather, it emerges as a consequence of
the Rayleigh turning geometry.

This suggests a natural relation between critical points of the exact
Rayleigh profile and spectral transitions of the linearized operator. In
the Kirchhoff case this question is especially interesting because the
linearization contains a nonlocal rank-one perturbation. The radial
simple-fold result above provides one instance in which the Rayleigh tangent
generates the relevant kernel mode and the detected turning bifurcation is
classified as a radial simple fold.

A complete stability theory requires further information. Along a spatial
branch $u_s(x)=U(sx)$, variational constraints and conserved quantities
generally change with $s$, so stability cannot simply be transferred from
the normalized local equation. The relation between turning bifurcations,
Morse-index changes, spectral crossings, and orbital stability therefore
remains a natural direction for further study; see
\cite{CazenaveLions1982,ZhangLiuSquassina2019}.

\section*{Acknowledgements}

During the preparation of this manuscript, ChatGPT (OpenAI) was used for
language editing and structural suggestions. All mathematical arguments,
calculations, references, and conclusions were independently checked and
verified by the author, who takes full responsibility for the final
manuscript.

\section*{Declarations}

\subsection*{Funding}

No funding was received for conducting this study.

\subsection*{Competing interests}

The author has no relevant financial or non-financial interests to disclose.

\subsection*{Data availability}

No datasets were generated or analysed during the current study.

\subsection*{Author contributions}

Yavdat Il'yasov: Conceptualization, methodology, formal analysis,
writing--original draft, writing--review and editing.

\end{document}